\documentclass{amsart}
\usepackage{a4wide,amsmath,amsthm,amssymb,amsfonts,stmaryrd}
\usepackage{mathdots}
\usepackage{latexsym}
\usepackage[all]{xy}
\usepackage{bm}
\usepackage{graphicx}
\usepackage{subfigure}
\usepackage{mathrsfs}
\usepackage{float}
\usepackage{makecell}
\usepackage{multirow}
\usepackage{xcolor}
\usepackage{cite}
\usepackage[pagebackref]{hyperref}
\usepackage{hyperref}
\usepackage{verbatim}
\usepackage{amsrefs}
\usepackage[toc,page]{appendix}

\numberwithin{equation}{section}
\newtheorem{theorem}{\bf{Theorem}}[section]
\newtheorem{proposition}[theorem]{\bf{Proposition}}

\newtheorem{corollary}[theorem]{\bf{Corollary}}
\newtheorem{lemma}[theorem]{\bf{Lemma}}
\newtheorem{remark}[theorem]{\bf{Remark}}
\newtheorem{example}[theorem]{\bf{Example}}

\newtheorem{assumption}[theorem]{\bf{Assumption}}

\newcommand{\abs}[1]{\left|#1\right|}

\newcommand{\car}[1]{\left|#1\right|}
\newcommand{\pairangone}[1]{\langle#1\rangle}

\newcommand{\rest}{\lvert}

\newcommand{\ord}[1]{\mathrm{ord}(#1)}

\newcommand{\mrgl}{\operatorname{GL}}
\newcommand{\mrsl}{\operatorname{SL}}
\newcommand{\mrsp}{\operatorname{Sp}}

\newcommand{\mrm}{\operatorname{M}}

\newcommand{\mrtr}{\operatorname{tr}}
\newcommand{\mrdet}{\operatorname{det}}
\newcommand{\mrdim}{\operatorname{dim}}
\newcommand{\mrhom}{\operatorname{Hom}}

\newcommand{\mrdiag}{\operatorname{diag}}
\newcommand{\mrgcd}{\operatorname{gcd}}

\newcommand{\mrrep}{\operatorname{Rep}}

\newcommand{\mrend}{\operatorname{End}}
\newcommand{\mraut}{\operatorname{Aut}}
\newcommand{\mrmod}{\operatorname{Mod}}

\newcommand{\mrInd}{\operatorname{Ind}}
\newcommand{\mrind}{\operatorname{ind}}

\newcommand{\mrinf}{\operatorname{Inf}}

\newcommand{\id}{\operatorname{id}}

\newcommand{\mrsupp}{\operatorname{Supp}}

\newcommand{\mbn}{\mathbb{N}}
\newcommand{\mbz}{\mathbb{Z}}
\newcommand{\mbc}{\mathbb{C}}

\newcommand{\mbf}{\mathbb{F}}

\newcommand{\mfp}{\mathfrak{p}}

\newcommand{\mfo}{\mathfrak{o}}
\newcommand{\mfa}{\mathfrak{a}}
\newcommand{\mfb}{\mathfrak{b}}

\newcommand{\mfs}{\mathfrak{s}}

\newcommand{\mfS}{\mathfrak{S}}

\newcommand{\mco}{\mathcal{O}}
\newcommand{\mcc}{\mathcal{C}}
\newcommand{\mcb}{\mathcal{B}}
\newcommand{\mca}{\mathcal{A}}

\newcommand{\mcf}{\mathcal{F}}
\newcommand{\mcg}{\mathcal{G}}
\newcommand{\mch}{\mathcal{H}}

\newcommand{\mcm}{\mathcal{M}}
\newcommand{\mcn}{\mathcal{N}}
\newcommand{\mcp}{\mathcal{P}}
\newcommand{\mcu}{\mathcal{U}}
\newcommand{\mcv}{\mathcal{V}}
\newcommand{\mcr}{\mathcal{R}}

\newcommand{\mcx}{\mathcal{X}}

\newcommand{\amax}{\mfa_{\text{max}}}

\newcommand{\bmax}{\mfb_{\text{max}}}
\newcommand{\Jmax}{J_{\text{max}}}

\newcommand{\Jonemax}{J_{\text{max}}^{1}}

\newcommand{\Honemax}{H_{\text{max}}^{1}}

\newcommand{\kappamax}{\kappa_{\mathrm{max}}}

\newcommand{\bs}[1]{\boldsymbol{#1}}

\newcommand{\ol}[1]{\overline{#1}}

\newcommand{\ul}[1]{\underline{#1}}

\begin{document}

	\title[Gelfand--Graev representation as a Hecke algebra module]{Gelfand--Graev representation as a Hecke algebra module of simple types of a finite central cover of $\mrgl(r)$}
	
	\author{Jiandi Zou}
	\address{Institute for Advanced Study in Mathematics of Harbin Institute of Technology, Harbin, China}
	\email{idealzjd@gmail.com}
	
	
	\begin{abstract}
		
		For an $n$-fold Kazhdan--Patterson cover or Savin's cover of a general linear group over a non-archimedean local field of residual characteristic $p$ with $\mrgcd(n,p)=1$, we realize the Gelfand--Graev representation as a Hecke algebra module of a simple type and study its explicit expression. As a main corollary, we calculate the Whittaker dimension of every discrete series representation of such a cover. Using Zelevinsky's classification, this theoretically gives the Whittaker dimension of every irreducible representation.
		
	\end{abstract}
	
	\maketitle
	\tableofcontents
	
	\section{Introduction}
	
	\subsection{Background}
	
	In the representation theory of $p$-adic  reductive groups, an important phenomenon is the existence and uniqueness of the Whittaker model, which has various applications such as constructing Rankin--Selberg L-functions, coordinating L-parameters in an L-packet and deducing different kinds of multiplicity one theorems, etc.
	
	More precisely, let $G$ be a quasi-split group over a non-archimedean local field $F$, let $U$ be a unipotent radical of a Borel subgroup of $G$ and let $\psi^{\vee}$ be a generic character of $U$. We call a complex smooth irreducible representation $\pi^{\vee}$ of $G$ \emph{generic} if the following vector space
	$$\mrhom_{U}(\pi^{\vee},\psi^{\vee})\cong\mrhom_G(\pi^{\vee},\mrInd_U^G(\psi^{\vee}))$$
	is non-zero, where $\mrInd_U^G(\psi^{\vee})$ denotes the induction and we used the Frobenius reciprocity. Then, it is known for a while (\emph{cf.} \cite{shalika1974multiplicity}*{Appendix}, also \cite{bernstein1976representations} and \cite{rodier1975modele} for $G=\mrgl_r(F)$) that the above space is one-dimensional. It means that there exists a unique way up to a scalar that embeds $\pi^{\vee}$ into $\mrInd_{U}^{G}(\psi^{\vee})$, identifying vectors in $\pi^{\vee}$ with  $(U,\psi^{\vee})$-equivariant complex smooth functions and justifing the word  ``Whittaker model". Taking the dual, the vector space
	$$\mrhom_{G}(\mrind_U^{G}(\psi),\pi)$$
	is also one-dimensional, where $\pi$ (resp. $\psi$) denotes the contragredient of $\pi^{\vee}$ (resp. $\psi^{\vee}$), and $\mrind_U^{G}(\psi)$ denotes the compact induction. Similarly, $\pi$ can be uniquely realized as a quotient of $\mrind_U^{G}(\psi)$, which is called the Gelfand--Graev model of $\pi$.
	
	Now we change our stage. Assume that $F$ contains all the $n$-th roots of unity denoted by $\mu_n$ for some fixed $n\geq 1$. Let $\ol{G}$ be an $n$-fold cover of $G$, i.e., a central extension of $G$ by $\mu_n$ as $\ell$-groups. In this case, there exists a unique splitting (\emph{cf.} Section \ref{sectionmetacover}) of $U$ that is invariant under the related Borel group conjugation (\emph{cf.} \cite{moeglin1995spectral}*{Appendix I}). Using this splitting, we identify $U$ with a subgroup of $\ol{G}$. Still, for an  irreducible (genuine) representation $\pi$ of $\ol{G}$, we would like to study the Whittaker dimension of $\pi$, or equivalently, the dimension of the space 
	$$\mrhom_{\ol{G}}(\mrind_U^{\ol{G}}(\psi),\pi).$$
	We remark that in general, unlike the linear case, this dimension is not always smaller than or equal to one.  
	Thus it is curious if for general enough representations the related Whittaker (resp. Gelfand--Graev) dimension could be determined. 
	
	We list the state-of-the-art of the above problem:
	\begin{itemize}
		\item Using Harish-Chandra's germ expansion, the dimension is always finite (\emph{cf.} \cite{patel2014theorem}).
		
		\item If $G$ is a torus, then the Whittaker dimension is exactly the dimension of the representation, which could be determined via the Stone--Von-Neumann theorem (\emph{cf.} for instance \cite{weissman2009metaplectic}).
		
		\item If $G$ is split and $\ol{G}$ is a tame Brylinski--Deligne cover\footnote{Here and after, by tameness of an $n$-fold cover of $G$ we mean that $\mrgcd(n,p)=1$.}, then for large classes of irreducible representations having Iwahori fixed vectors, Gao \cite{gao2023r} proposed an explicit numerical conjecture, which has been proved by Gao--Gurevich--Karasievicz \cite{gao2024genuine} by realizing the Gelfand--Graev representation as a pro-$p$ Iwahori Hecke algebra module.
		
		\item If $\ol{G}$ is a tame Brylinski--Deligne cover, then for depth-0 cuspidal representations, Gao--Weissman \cite{gao2019whittaker} calculated the dimension, generalizing the previous result of Blondel \cite{blondel1992uniqueness} when $G=\mrgl_r(F)$.
		
		\item If $\ol{G}$ is a tame Kazhdan--Patterson cover of $G=\mrgl_r(F)$,  the author \cite{zou2022metaplectic} proposed a conjecture (\cite{zou2022metaplectic}*{Conjecture 4.6}, that was verified for certain cases) related to the special values of the Harish-Chandra character of a generalized Speh representation of $G$. Based on that, the author calculated the related dimension for discrete series representations of $\ol{G}$ using the metaplectic correspondence.     
		
	\end{itemize}
	However, because of the failure of multiplicity one in general, the usual method of considering invariant distributions in the proof of Shalika \cite{shalika1974multiplicity} does not work anymore, making the calculation for general representations hard enough.
	
	\subsection{Main theorem and applications}
	
	The main goal of our paper is two-fold:
	
	\begin{itemize}
		\item One the one hand, we calculate the Whittaker dimension for general irreducible representations of a tame Kazhdan--Patterson cover or Savin's cover $\ol{G}$ of $G=\mrgl_r(F)$ (\emph{cf.} \S \ref{subsectionKPScover});
		
		\item On the other hand, the method we developed is general, which could be used to calculate the Whittaker dimension for other covers, once the related ``type theory" is developed.
	\end{itemize}  
	To explain our result, we assume the readers to have a basic knowledge on the simple type theory of $\ol{G}$ developed in \cite{zou2023simple} and leave Section \ref{sectionsimpletype} for more details. Let $(\ol{J},\lambda)$ be a simple type of $\ol{G}$, where $\ol{J}$ is an open compact subgroup of $\ol{G}$ and $\lambda$ an irreducible representation of $\ol{J}$. Then, we may define the following functor:
	$$\bs{\mathrm{M}}_\lambda:\mrrep(\ol{G})\rightarrow\mrmod(\mch(\ol{G},\lambda)),\quad \pi\mapsto\mrhom_{\ol{G}}(\mrind_{\ol{J}}^{\ol{G}}(\lambda),\pi),$$
	where $\mch=\mch(\ol{G},\lambda)$ is the related Hecke algebra, $\mrrep(\ol{G})$ the category of smooth representations of $\ol{G}$ and $\mrmod(\mch(\ol{G},\lambda))$ the category of left $\mch(\ol{G},\lambda)$-modules. In particular, $\mch=\mca\otimes_{\mbc}\mch_0$ is an affine Hecke algebra of type A equipped with Bernstein's presentation, where $\mca$ is the ``lattice part" as a free abelian group algebra of rank $k$ (a positive integer depending on $\lambda$) and $\mch_0$ is the ``finite part" as a finite Hecke algebra related to the symmetric group $\mfS_k$.
	
	Now we plug the Gelfand--Graev representation $\mcv=\mrind_U^{\ol{G}}(\psi)$ into that functor. The following main theorem determines the structure of  $\mcv^{\lambda}:=\bs{\mathrm{M}}_\lambda(\mcv)$ as an  $\mch$-module.
	
	\begin{theorem}[\emph{cf.} Theorem \ref{thmmain}]\label{thmmainone}
		
		Let $\ol{G}$ be either a tame Kazhdan--Patterson cover or the Savin cover of $G=\mrgl_r(F)$. Then we have a decomposition of $\mch$-modules $$\mcv^{\lambda}=\bigoplus_{\mco\in\mcx(\lambda)/\mfS_k }\mcv_{\mco}^{\lambda},$$
		such that each $\mcv_{\mco}^{\lambda}$ is isomorphic to  $\mca\otimes_{\mbc}(\mch_{0}\otimes_{\mch_{\mco}}\varepsilon_{\mco})$, where
		\begin{itemize}
			\item  $\mcx(\lambda)$ is a finite abelian group constructed from $\lambda$ and endowed with an $\mfS_k$-action  (\emph{cf.} \S \ref{subsectionsimpletypes});
			
			\item $\mch_{\mco}$ is the Hecke subalgebra of $\mch_0$ related to the stabilizer of $\mco$ in $\mfS_k$;
			
			\item $\varepsilon_{\mco}$ is the sign character of $\mch_{\mco}$, thus $\mch_{0}\otimes_{\mch_{\mco}}\varepsilon_{\mco}$ is an $\mch_{0}$-module.
			
		\end{itemize}
		
	\end{theorem}
	
	The main corollary of the above theorem is the calculation of the Gelfand--Graev dimension. Let $\pi$ be a discrete series representation of $\ol{G}$ whose inertial equivalence class $\mfs$ corresponds to the simple type $(\ol{J},\lambda)$, or in other words, $(\ol{J},\lambda)$ is a type of $\mfs$ in the sense of Bushnell--Kutzko. Then, applying the functor $\bs{\mathrm{M}}_\lambda$, we get the following isomorphism of vector spaces:
	$$\mrhom_{\ol{G}}(\mcv,\pi)\cong\mrhom_{\mch}(\mcv^\lambda,\pi^\lambda).$$
	We already know the structure of $\mcv^\lambda$. On the other hand,  $\pi^\lambda$ is a one-dimension module of $\mch$ whose restriction to $\mch_0$ is the sign character $\varepsilon_{0}$. Then, after standard argument,
	
	\begin{corollary}[\emph{cf.} Proposition \ref{propdiscreteWhittaker}]
		
		We have $\dim_{\mbc}\mrhom_{\ol{G}}(\mrind_U^{\ol{G}}(\psi),\pi)=\car{\mcx(\lambda)/\mfS_k}$.
		
	\end{corollary}
	
	Since for every discrete series representation $\pi$ we may find a simple type $(\ol{J},\lambda)$ as above, then we get the Gelfand--Graev dimension of any discrete series representation $\pi$. Using Zelevinsky's classification, theoretically we get the Gelfand--Graev dimension (or equivalently, Whittaker dimension) of every irreducible representation of $\ol{G}$.
	
	We emphasize several applications/motivations of our result.
	
	\begin{itemize}
		
		\item In the doubling method of constructing L-functions of covers of general linear groups and symplectic groups, it is crucial to show that certain degenerate Whittaker model of certain generalized Speh representations have the multiplicity one (\emph{cf.} \cite{kaplan2019doubling}, \cite{kaplan2022rankin}*{Conjecture 14}). This could be done by combining our dimensional result with a standard Bernstein--Zelevinsky style of argument. We leave \cite{zou2022metaplectic}*{Section 5} for more details.
		
		\item It is conjectured \cite{wang2024distinction}*{Conjecture 1.1} that the multiplicity of a tempered representation $\pi$ of $G=\mrgl_r(F)$ occurring in the $G$-space of smooth functions on invertible symmetric matrices can be expressed as the sum of the square of the Whittaker dimension of those tempered representations $\ol{\pi}$ of a two-fold Kazhdan--Patterson cover $\ol{G}$ of $G$ that lifts to $\pi$ via the metaplectic correspondence. Then, our result provides numerical evidence for that conjecture, which should be regarded as an ``exceptional" case of the so-called ``relative local Langlands program".
		
		\item In \cite{gao2022gelfand}, the authors used the decomposition of the Gelfand--Graev representation (\emph{cf.} \cite{gao2024genuine}*{Theorem 1.2}) as a bridge to construct a type A quantum Schur--Weyl duality. Since our Theorem \ref{thmmainone} generalizes their decomposition theorem in the type A case, it is expected that their study of quantum Schur--Weyl duality could be considered for general simple types as well.
		
	\end{itemize} 
	
	\subsection{Sketch of proof}
	
	We sketch the proof of Theorem \ref{thmmainone}. 
	
	For those readers familiar with \cite{gao2024genuine}, it is not hard to see that the statement of our main results largely follows from theirs. However, unlike \emph{loc. cit.} where they considered the pro-$p$ Iwahori Hecke algebra action, it seem difficult for the author to find an analogue of pro-$p$ Iwahori Hecke algebra for a general simple type. 
	
	Instead, we use the following three-step strategy which is more or less influenced by the arguments in \cite{chan2018iwahori} and \cite{chan2019bernstein}. 
	
	\begin{enumerate}
		
		\item We first show that $\mcv^{\lambda}$ is a free $\mca$-module of rank $\car{\mcx(\lambda)}$.
		
		\item We may pick a certain sequence of representatives $t_1,\dots,t_d\in \mcx(\lambda)$ of $\mcx(\lambda)/\mfS_k$ with $d=\car{\mcx(\lambda)/\mfS_k}$ and related elements $\Phi_{t_1},\dots,\Phi_{t_d}\in \mcv^{\lambda}$ with $t_{i}\in \mcx(\lambda)$, such that the $\mch_0$-module $\mch_0\ast\Phi_{t_i}$ is isomorphic to $\mch_0\otimes_{\mch_{\mco_{t_i}}}\varepsilon_{\mco_{t_i}}$ for each $i$. 
		
		\item Finally, let $\mcb(\mch_0\ast\Phi_{t_i})$ be a $\mbc$-basis of the vector space $\mch_0\ast\Phi_{t_i}$ for each $i$. We show that $\bigcup_{i=1}^d\mcb(\mch_0\ast\Phi_{t_i})$ forms a free $\mca$-basis of $\mcv^{\lambda}$, then we have $$\mcv^{\lambda}=\bigoplus_{i=1}^t\mcv_{\mco_{t_i}}^{\lambda},$$where $\mcv_{\mco_{t_i}}^{\lambda}:=(\mca\otimes_{\mbc}\mch_0)\ast\Phi_{t_i}$ for $i=1,\dots,t$.
		
	\end{enumerate}
	
	Bootstrapping the above three steps requires more notations and technical details in the simple type theory, which could be too heavy for the readers. Instead, we only mention some key ingredients that will be used in the proof.
	\begin{itemize}
		
		\item 
		(\emph{cf.} \S \ref{subsectionPSresult}) We may pick the generic pair $(U,\psi)$ that is of ``good" relative position with respect to $\lambda$. 
		
		\item (\emph{cf.} \S \ref{subsectionKJacofGG}) Consider the parabolic subgroup $\ol{P}=\ol{M}N$ of $\ol{G}$ related to $\lambda$ and the  Gelfand--Graev model $\mcv_{M}=\mrind_{U\cap \ol{M}}^{\ol{M}}(\psi\rest_{U\cap \ol{M}})$ of $\ol{M}$. Then the normalized Jacquet functor induces an isomorphism $r_{N}(\mcv)\cong\mcv_M$ of $\ol{M}$-representations. Moreover, for a certain element $\Phi$ in $\mcv$ having ``good support", the support of $r_N(\Phi)$ in $\mcv_M$ can also be controlled.
		
		\item Step (1) relates to the cuspidal case, which will be dealt in Section \ref{sectioncuspidal}.
		
		\item Step (2) concerns the relation between the Gelfand--Graev representation of a non-archimedean local general linear group and that of a finite general linear group, which will be dealt in Section \ref{sectionH0module}. Also, most delicate type theoretical arguments will be done here.
		
		\item Finally, Step (3)  will be done in Section \ref{sectionHmodulemcv}. To show the linear independence of certain elements in $\mcv^{\lambda}$, our strategy is to show that their ``supports", which are unions of $U$-$\ol{J}$ double cosets, are disjoint.
		
	\end{itemize}
	We emphasize that our proof relies on the original multiplicity one result of Whittaker model for general linear groups, which has already been used in the work of Paskunas--Stevens \cite{paskunas2008realization}. Thus when $n=1$ we do not give a new proof for the multiplicity one result for general linear groups, but rather we build up the relation between the Whittaker dimension of representations of a cover and that of a linear group. We wonder if an essential new proof of the multiplicity one of the Whittaker model could be given for a non-archimedean general linear group. For essentially tame supercuspidal representations, such a proof is somehow given by Tam \cite{tam2019explicit} using simple type theory.

	Finally, our strategy is general that has a potential of being used to the covers of other quasi-split groups, once the related type theory is developed. However, there should be some restrictions on the covers we are considering. When $G=\mrgl_r(F)$, we only consider a Kazhdan--Patterson cover or the Savin cover, although we have a simple type theory for a general Brylinski--Deligne cover. There are two reasons. First, the decomposition $\mch=\mca\otimes_{\mbc}\mch_0$ for the Hecke algebra $\mch$ is required (more precisely, Assumption \ref{assumHeckealg}), which is proved only for the above two classes of covers. Secondly, to realize Step (3) (more precisely, strategy \eqref{eqstrategy}), a more delicate structure of the cover is required, which is verified only for the above two classes of cover.
	
	\subsection{Acknowledgement}
	
	The author would like to thank Vincent S\'echerre for pointing out the reference \cite{secherre2016block}, especially Lemma 4.2 and Statement (5.6) in \emph{ibid.}, which helps the author to solve the last technical puzzle Lemma \ref{lemmatechnical} of this paper. The author would also like to thank Fan Gao, Max Gurevich and Edmund Karasiewicz for enlightening discussions. This work was partially supported by the project PAT4832423 of the Austrian Science Fund (FWF).

	\section{Notations}
	
	In this article, we fix a non-archimedean locally compact field $F$, whose residual field $\bs{k}$ is of cardinality $q$. We write $\mfo_{F}$ for the ring of integers of $F$ and $\mfp_{F}$ the maximal ideal of $\mfo_{F}$. 
	Let $\abs{\cdot}_{F}$ denote the normalized discrete valuation of $F$.
	
	We fix a positive integer $n$, such that the subgroup of $n$-th roots of unity of $F^{\times}$, denoted by $\mu_{n}(F)$ and usually abbreviated by $\mu_{n}$, consists of $n$'s elements. Beginning from \S \ref{subsectionmetacoverGL}, we assume that $n$ divides $q-1$. In other words, every cover we are going to consider is ``tame".	
	
	We denote by $(\cdot,\cdot)_{n}:F^{\times}\times F^{\times}\rightarrow\mu_{n}$ the $n$-th Hilbert symbol, which is a bimultiplicative, anti-symmetric pairing, that descends to a non-degenerate bimultiplicative pairing $F^{\times}/F^{\times n}\times F^{\times}/F^{\times n}\rightarrow\mu_{n}$, where $F^{\times n}=\{x^{n}\mid x\in F^{\times}\}$. If $\mrgcd(q,n)=1$, such a pairing is trivial on $\mfo_{F}^{\times}\times\mfo_{F}^{\times}$. 
	
	By $\ell$-groups in this article, we mean locally compact totally disconnected topological groups as in \cite{bernstein1976representations}. By representations of an $\ell$-group, we mean complex smooth representations. In particular, a character is a one-dimensional smooth representation. 
	
	
	Let $G$ be an $\ell$-group. We write $Z(G)$ for the center of $G$ and $[\cdot,\cdot]:G\times G\rightarrow G,\ (g_{1},g_{2})\mapsto[g_{1},g_{2}]:=g_{1}g_{2}g_{1}^{-1}g_{2}^{-1}$ the commutator map.
	
	
	
	Let $H\subset H'$ be two closed subgroups of $G$. We write $\mrInd_{H}^{H'}:\mrrep(H)\rightarrow\mrrep(H')$, $\mrind_{H}^{H'}:\mrrep(H)\rightarrow\mrrep(H')$ and $\rest_{H}:\mrrep(H')\rightarrow\mrrep(H)$ for the induction, compact induction and restriction functors respectively, where $\mrrep(\cdot)$ denotes the category of smooth representations. For a representation $\pi$ of $H'$, we often write $\pi$ instead of $\pi|_{H}$ for short, if the domain of definition of $\pi$ (i.e. $H$) is clear from the context. In convention, we say that a representation $\pi$ of $H'$ contains an irreducible representation $\pi'$ of $H$ if $\mrhom_{H}(\pi',\pi\rest_{H})\neq 0$. Assume $H$ to be an open normal subgroup of $H'$. Let $\ul{H}=H'/H$ and let $\varrho$ be a representation of $\ul{H}$. Then we denote by $\mrinf_{\ul{H}}^{H'}\varrho$ the \emph{inflation} of $\varrho$ as a representation of $H'$.
	
	Let $\rho$ be a representation of $H$. 
	For $g,x\in G$, we write $x^{g}=g^{-1}xg$, $H^{g}=g^{-1}Hg$, and $\rho^{g}(\cdot)=\rho(g\cdot g^{-1})$ for a representation of $H^{g}$. 
	We say that $g\in G$ \emph{intertwines} $\rho$, if the intertwining space
	$$\mrhom_{H^{g}\cap H}(\rho^{g},\rho)$$
	is non-zero, and we denote by $I_{G}(\rho)$ the intertwining set of $\rho$ consisting of $g\in G$ whose related intertwining space is non-zero. 
	We say that $g$ \emph{normalizes} $H$ (resp. $\rho$) if $H^{g}=H$ (resp. $\rho^{g}\cong\rho$), and we write $N_{G}(H)$ (resp. $N_{G}(\rho)$) for the corresponding normalizer. 	
	

	\section{Metaplectic covers of $\mrgl_{r}(F)$}\label{sectionmetacover}
	
	\subsection{General theory for a finite central cover}
	
	We refer to \cite{gan2018groups} and \cite{zou2023simple} for the basic settings. Let $G$ be an $\ell$-group. By an $n$-fold cover of $G$, we mean a central extension of $G$ by $\mu_{n}$ as $\ell$-groups:
	\begin{equation}\label{eqcentralext}
		\xymatrix{1 \ar[r] & \mu_{n} \ar[r]^-{} & \ol{G} \ar[r]^-{\bs{p}} & G \ar[r] & 1}.
	\end{equation}
	The set of equivalence classes of central extensions is in bijection with the set of continuous 2-cohomology classes $H^{2}(G,\mu_{n})$. More precisely, given an $n$-fold cover $\ol{G}$ of $G$ and a continuous map $\bs{s}:G\rightarrow G$ such that $\bs{p}\circ\bs{s}=\id$, the corresponding cohomology class is related to the 2-cocycle $\sigma:G\times G\rightarrow \mu_{n}$ satisfying
	\begin{equation}\label{eq2cocycle1}
		\bs{s}(g_{1})\bs{s}(g_{2})\sigma(g_{1},g_{2})=\bs{s}(g_{1}g_{2}),\quad g_{1},g_{2}\in G.
	\end{equation}
	It is clear that for two covers $\ol{G}_{1}$ and $\ol{G}_{2}$ of $G$ realized by 2-cocycles $\sigma_{1}$ and $\sigma_{2}$ respectively, the Baer sum of $\ol{G}_{1}$ and $\ol{G}_{2}$ is realized by the 2-cocycle $\sigma_{1}\cdot\sigma_{2}$.
	
	In convention, for a closed subgroup $H$ of $G$ we write $\ol{H}$ for the preimage $\bs{p}^{-1}(H)$ in $\ol{G}$. We identify $\mu_{n}$ with a central subgroup of $\ol{G}$. 
	
	A \emph{splitting} of $H$ is a continuous group homomorphism $\bs{s}_{H}:H\rightarrow \ol{G}$ satisfying $\bs{p}\circ\bs{s}_{H}=\id$. In general, a splitting may neither exist nor be unique. If $H$ is a pro-$p$-group and $\mrgcd(n,p)=1$, then the cohomology group $H^{2}(H,\mu_{n})$ is trivial, and there exists a unique splitting $\bs{s}_{H}$ of $H$. In this case, we identify $H$ with the subgroup $\bs{s}_{H}(H)$ of $\ol{H}$. For a representation $\rho$ of $H$, we similarly identify it with the corresponding representation of $\bs{s}_{H}(H)$. In general, even if there does not exist a unique splitting of $H$, the above convention still works once we fix a splitting $\bs{s}_{H}$. 
	
	The commutator $[\cdot,\cdot]:\ol{G}\times \ol{G}\rightarrow \ol{G}$ factors through $G\times G$, and we denote by $[\cdot,\cdot]_{\sim}:G\times G\rightarrow \ol{G}$ the resulting map. In particular, if $g_{1},g_{2}\in G$ are commutative, then  
	\begin{equation}\label{eqformcommutator}
		[g_1,g_2]_{\sim}=\sigma(g_1,g_2)\sigma(g_2,g_1)^{-1}\in\mu_{n}.
	\end{equation} 
	Similarly, the $\ol{G}$-conjugation on $\ol{G}$ factors through $G$, so we may consider the $G$-conjugation on $\ol{G}$. It is clear that $[g_{1}^{g},g_{2}^{g}]_{\sim}=[g_{1},g_{2}]_{\sim}$ for any $g,g_{1},g_{2}\in G$. Also, for closed subgroups $H,H_{1},H_{2}$ of $G$, we define the coset $g\ol{H}:=\bs{s}(g)\ol{H}$ and the double coset $\ol{H_{1}}g\ol{H_{2}}:=\ol{H_{1}}\bs{s}(g)\ol{H_{2}}$, which does not depend on the choice of the section $\bs{s}$.
	
	
	Fix a faithful character $\epsilon:\mu_{n}\rightarrow\mbc^{\times}$. A representation $\pi$ of $\ol{G}$ is called ($\epsilon$-)\emph{genuine} if $\mu_{n}$ acts by $\epsilon$ on $\pi$. 
	
	For a representation $\rho$ and a splitting $\bs{s}$ of $H$, we write $\epsilon\cdot\rho$ (resp. $1_{\mu_n}\cdot \rho $) for the extension of $\rho$ as an $\epsilon$-genuine (resp. non-genuine) representation of $\ol{H}$.
	
	For a representation $\chi$ of $H$ and a genuine representation $\rho$ of $\ol{H}$, the tensor product $\rho\otimes\chi:=\rho\otimes(\chi\circ\bs{p})$ is well-defined as a genuine representation of $\ol{H}$. When $\chi$ is a character, we write $\rho\chi$ or $\rho\cdot\chi$ instead.
	
	\subsection{Metaplectic covers of $\mrgl_{r}(F)$}\label{subsectionmetacoverGL}
	
	From now on, we assume that $n$ divides $q-1$. We consider $n$-fold covers of $G=\mrgl_{r}(F)$ arising from Brylinski--Deligne's construction (\emph{cf.} \cite{brylinski2001central}). 
	
	First we consider the $n$-fold cover $\ol{G}_{\mrdet}$ related to the 2-cocycle
	\begin{equation}\label{eq2cocycledet}
		\sigma_{\mrdet}:G\times G\rightarrow\mu_{n},\quad (g_{1},g_{2})\mapsto(\mrdet(g_{1}),\mrdet(g_{2}))_{n},
	\end{equation}
	called the determinant cover.
	
	We further consider the canonical $n$-fold cover of $\mrsl_{r+1}(F)$ considered by Matsumoto with respect to the Steinberg symbol $(\cdot,\cdot)_{n}^{-1}:F^{\times}\times F^{\times}\rightarrow \mu_{n}$ (\emph{cf.} \cite{matsumoto1969sous}). Via the pull-back of the embedding $$\mrgl_{r}(F)\rightarrow\mrsl_{r+1}(F),\quad g\mapsto(\mrdet(g)^{-1},g),$$ we get an $n$-fold cover of $G$, which we denote by $\ol{G}_{\mathrm{KP}}$ and call it the standard Kazhdan--Patterson cover. Moreover in \cite{kazhdan1984metaplectic}*{\S 0.I}, \cite{banks1999block}, a special 2-cocycle $\sigma_{\mathrm{KP}}:G\times G\rightarrow\mu_{n}$ related to the cover $\ol{G}_{\mathrm{KP}}$ is constructed for each $r$. It is trivial when $r=1$ and satisfies the following block-compatibility:
	\begin{equation}\label{eqBLSblock}
		\begin{aligned}
			\sigma_{\mathrm{KP}}(\mrdiag(g_1,\dots,g_k),\mrdiag(g_1',\dots,g_k'))=
			\bigg[\prod_{i=1}^k\sigma_{\mathrm{KP}}(g_i,g_i')\bigg]\cdot\bigg[\prod_{1\leq i<j\leq k}(\det(g_i),\det(g_j'))_n\bigg],   \end{aligned}
	\end{equation} 
	where $g_{i},g_{i}'\in \mrgl_{r_{i}}(F)$ for each $i=1,\dots,k$ and $r=r_{1}+\cdots+r_{k}$.
	
	In general, the Baer sum of copies of $\ol{G}_{\mrdet}$ and $\ol{G}_{\mathrm{KP}}$ ranges over all the $n$-fold covers of $G$ arising from Brylinski--Deligne's construction. In other words, when $(\bs{c},\bs{d})$ ranges over $\mbz/n\mbz\times\mbz/n\mbz$, the corresponding 2-cocycles $\sigma_{\mrdet}^{\bs{c}}\cdot\sigma_{\mathrm{KP}}^{\bs{d}}$ represent all the Brylinski--Deligne $n$-fold covers of $G$.

	From now on, we fix $r\geq 1$, $(\bs{c},\bs{d})\in\mbz\times\mbz$ and the related 2-cocycle $\sigma=\sigma_{\mrdet}^{\bs{c}}\cdot\sigma_{\mathrm{KP}}^{\bs{d}}$. Let $\ol{G}$ be the $n$-fold metaplectic cover of $G=\mrgl_{r}(F)$ corresponding to $\sigma$. 
	
	For a given maximal open compact subgroup $K$ of $G$, we fix a splitting $\bs{s}_{K}:K\rightarrow \ol{G}$, which is possible since $\mrgcd(n,q)=1$ (\emph{cf.} \cite{zou2023simple}*{Proposition 4.2.(6)}).
	
	
	Let $P=MN$ be a parabolic subgroup of $G$, where $M$ is a Levi factor and $N$ is the related unipotent radical. Indeed, we have $M\cong G_{r_{1}}\times\dots\times G_{r_{k}}$ for some composition $r=r_{1}+\dots+r_{k}$, where each $G_{r_{i}}$ is isomorphic to $\mrgl_{r_{i}}(F)$. Since $N$ is a pro-$p$-group, there exists a unique splitting of $N$ into $\ol{G}$. Thus we may realize $N$ as a subgroup of $\ol{G}$. So $\ol{P}=\ol{M}N$
	is a parabolic subgroup of $\ol{G}$. 
	
	
	Let $H=H_{1}\times\dots\times H_{k}$ be a subgroup of $M$, such that each $H_{i}$ is a closed subgroup of $G_{r_{i}}$. We call $H$ \emph{block compatible} if $[H_{i},H_{j}]_{\sim}=\{1\}$ for any  $1\leq i<j\leq k$. Indeed, $H$ is block compatible in the following cases:
	\begin{itemize}
		\item For every $i$, the determinant of every element in $H_{i}$ is an $n$-th power in $F^{\times}$;
		\item For every $i$, the determinant of every element in $H_{i}$ is in $\mfo_{F}^{\times}$.
	\end{itemize}
	Let $\rho_{i}$ be a genuine representation of $\ol{H_{i}}$ for each $i$. We take the tensor product $\rho_{1}\boxtimes\dots\boxtimes\rho_{k}$ as a representation of $\ol{H_{1}}\times\dots\times \ol{H_{k}}$, trivial on $$\Xi=\{(\zeta_{1},\dots,\zeta_{k})\in\mu_{n}\times\dots\times\mu_{n}\mid \zeta_{1}\dots\zeta_{k}=1\}.$$
	If $H$ is block compatible, then it is clear that $$\ol{H}\cong \ol{H_{1}}\times\dots\times \ol{H_{k}}/\Xi,$$
	so we realize $\rho_{1}\boxtimes\dots\boxtimes\rho_{k}$  as a genuine representation of $\ol{H}$. 
	
	\subsection{KP-covers and S-cover}\label{subsectionKPScover}
	
	We emphasize the following two special classes of covers.
	
	\begin{itemize}
		
		\item (The Kazhdan--Patterson covers) When $\bs{d}=1$, we get Kazhdan--Patterson covers 
		(\emph{cf.} \cite{kazhdan1984metaplectic}) Such covers could be regarded as the Baer sum of the standard Kazhdan--Patterson cover and $\bs{c}$-copies of the determinantal cover. We denote by $\sigma_{\mathrm{KP},\bs{c}}$ the related 2-cocycle and $\ol{G}_{\mathrm{KP}}^{\bs{c}}$ the related $n$-fold cover.
		
		\item (Savin's cover) When $\bs{c}=-1$ and $\bs{d}=2$, we get a special cover constructed by Gordan Savin. 
		It can be realized as follows: first we construct the canonical cover of the symplectic group $\mrsp_{2r}(F)$ with respect to the Steinberg symbol $(\cdot,\cdot)_{n}^{-1}$, and then after identifying $G$ with the Siegel Levi subgroup of $\mrsp_{2r}(F)$ we get Savin's cover of $G$. We denote by $\sigma_{\mathrm{Sav}}$ the related 2-cocycle and $\ol{G}_{\mathrm{Sav}}$ the related $n$-fold cover. 
		A special feature of  Savin's cover is that every Levi subgroup of $G$ is block compatible.
		
	\end{itemize} 
	
	\begin{remark}
		
		For the above two classes of covers, we will take advantage of the following privileges:
		
		\begin{itemize}
			\item We have the so-called metaplectic tensor product functor and Zelevinsky's classification for genuine irreducible representations (\emph{cf.} \cite{kaplan2022classification}).
			
			\item We have a detailed description of the Hecke algebra of a simple type (\emph{cf.} Assumption \ref{assumHeckealg}).
			
			\item The ``lattice part" of the Hecke algebra as well as the Gelfand--Graev module are explicit enough to be compared (\emph{cf.} Section \ref{sectionHmodulemcv}).
			
		\end{itemize}
		
	\end{remark}
	
	\section{Simple type theory}\label{sectionsimpletype}
	
	In this part, we briefly recall the simple type theory introduced in \cite{zou2023simple}*{\S 5, \S 6}. We also recommend \cite{bushnell129admissible} for the original theory for general linear groups. 
	
	\subsection{Strata, groups and simple characters}
	
	Let $V$ be an $r$-dimensional vector space over $F$. Let $A=\mrend_F(V)\cong\mrm_{r}(F)$ and $G=\mraut_{F}(V)\cong\mrgl_r(F)$. 
	
	A simple stratum $[\mfa,u,0,\beta]$ consists of a hereditary order $\mfa$ (i.e. the endomorphism algebra of an $\mfo_{F}$ lattice chain), a non-negative integer $u$,  and an element $\beta$ in $A$, such that
	$E=F[\beta]$ is a field over $F$ that normalizes $\mfa$. Here, $-u$ equals the valuation of $\beta$ with respect to $\mfa$ (i.e. $\beta\mfa=\mfp_{\mfa}^{-u}$, where $\mfp_{\mfa}$ is the prime ideal of $\mfa$.). 
	
	Let $B$ be the centralizer of $E$ in $A$, which is an $E$-algebra isomorphic to $\mrm_{m}(E)$. Let $\mfb=B\cap\mfa$ which is a hereditary order in $B$, where $d=[E:F]$ and $m=r/d$. Write $\bs{l}$ for the residue field of $E$. In general, there exists a composition $m=m_1+\dots+m_k$, such that $\mfb$ is isomorphic to the $\mfo_{E}$-subalgebra 
	\begin{equation}\label{eqmfb}
		\{(b_{ij})_{1\leq i,j\leq k}\mid b_{ij}\in \mrm_{m_i\times m_j}(\mfo_E),\ i\leq j\ \text{and}\ b_{ij}\in \mrm_{m_i\times m_j}(\mfp_E),\ i>j\} 
	\end{equation}
	of $\mrm_{m}(\mfo_{E})$. If $k=1$, or in other words $\mfb$ is a maximal $\mfo_E$-order in $B$, we say that the simple stratum we are considering is maximal.
	
	Let $U(\mfa)=\mfa^\times$ (resp. $U(\mfb)=\mfb^\times$) and $U^1(\mfa)$ (resp. $U^1(\mfb)$) be its pro-unipotent radical.
	
	In convention, we have $u=0$ if and only if $\mfa$ is maximal and $\beta\in \mfo_{F}^{\times}$, meaning that we consider a null simple stratum. In this case we have $E=F$ and $m=r$.
	
	Associated to $[\mfa,u,0,\beta]$, a sequence of subgroups
	$$H^{1}(\beta,\mfa)\subset J^{1}(\beta,\mfa)\subset J(\beta,\mfa)\subset\bs{J}(\beta,\mfa)$$
	of $G$ are constructed. Here, both $H^{1}(\beta,\mfa)$ and $J^{1}(\beta,\mfa)$ are open compact pro-$p$-groups, and $J(\beta,\mfa)=U(\mfb) J^{1}(\beta,\mfa)$ is an open compact subgroup of $U(\mfa)\subset G$. In particular, we have $$J(\beta,\mfa)/J^{1}(\beta,\mfa)\cong U(\mfb)/U^1(\mfb),$$
	which is isomorphic to the Levi subgroup 
	$$\mcm\cong\mrgl_{m_1}(\bs{l})\times\dots\times\mrgl_{m_k}(\bs{l})$$ of $\mcg:=\mrgl_{m}(\bs{l})$. We denote by $\bs{J}(\beta,\mfa)$ the normalizer of $J(\beta,\mfa)$ in $G$. If the simple stratum is maximal, in particular we have $\bs{J}(\beta,\mfa)=E^\times J(\beta,\mfa)$. 
	
	Now we consider representations. Associated to $[\mfa,u,0,\beta]$ and a fixed additive character $\psi_F$ of $F$, there is a finite set $\mcc(\mfa,0,\beta)$ of characters of $H^{1}(\beta,\mfa)$, called simple characters. Let $\theta$ be a simple character of $H^{1}(\beta,\mfa)$. There exists a unique irreducible representation $\eta$ of $J^{1}(\beta,\mfa)$ containing $\theta$, called the Heisenberg representation of $\theta$. We further have $I_{G}(\theta)=I_{G}(\eta)=B^{\times}JB^{\times}$. We may further find an irreducible representation $\kappa$ of $J(\beta,\mfa)$ extending $\eta$, such that $I_{G}(\kappa)=I_{G}(\eta)$. Such $\kappa$ is called a $\beta$-extension of $\eta$.
	
	We may find a decomposition of $E$-vector spaces $V=\bigoplus_{i=1}^{k}V^{i}$ and an $E$-basis $\{v_{i1},\dots,v_{im_{i}}\}$ of $V^{i}$ for each $i$, such that 
	\begin{itemize}
		\item the lattice chains of $\mfa$ and $\mfb$ are decomposable with respect to the above decomposition;
		\item under the basis $\{v_1',\dots ,v_m'\}:=\{v_{11},\dots,v_{1m_{1}},\dots,v_{k1},\dots,v_{km_{k}}\}$ of $V$, the $E$-algebra $B$ is identified with $\mrm_m(E)$ and the $\mfo_E$-order $\mfb$ is identified with the order \eqref{eqmfb}.
	\end{itemize}
	
	We construct a parabolic subgroup $P=MN$ of $G$, where
	$$P=\big(\prod_{1\leq i\leq j\leq k}\mrhom_{F}(V^i,V^j)\big)\cap G,\ M=\prod_{1\leq i\leq k}\mraut_{F}(V^i),\ N=\big(\id+\prod_{1\leq i< j\leq k}\mrhom_{F}(V^i,V^j)\big)\cap G.$$
	In other words, $P$ is the parabolic subgroup that fixes the flag
	$$\{0\}\subset V^1\subset V^1\oplus V^2\subset\dots\subset  V^1\oplus\dots\oplus V^k=V.$$
	Let $P^-$ (resp. $N^-$) be the opposite of $P$ (resp. $N$). 
	
	Let $A^i=\mrend_{F}(V^i)$, $G^i=\mraut_F(V^i)$, $B^i=\mrend_{E}(V^i)$, $\mfa^i=A^i\cap \mfa$ and $\mfb^i=B^i\cap \mfb$. Then, each $[\mfa^i,u_i,0,\beta]$ is a maximal simple stratum in $A^i$, where $-u_i$ denotes the valuation of $\beta$ with respect to $\mfa^i$. For each $i$, we define related subgroups $H^1(\beta,\mfa^i), J^1(\beta,\mfa^i)$ and $J(\beta,\mfa^i)$ of $G^i$. Then we have decompositions
	\begin{equation*}
		\begin{aligned}
			H_M^1(\beta,\mfa):=H^1(\beta,\mfa)\cap M=&H^1(\beta,\mfa^1)\times\dots\times H^1(\beta,\mfa^k),\quad \theta_M:=\theta\rest_{H_M^1(\beta,\mfa)}=\theta_1\boxtimes\dots\boxtimes \theta_k\\
			J_M^1(\beta,\mfa):=J^1(\beta,\mfa)\cap M=&J^1(\beta,\mfa^1)\times\dots\times J^1(\beta,\mfa^k),\quad \eta_M:=\eta\rest_{J_M^1(\beta,\mfa)}=\eta_1\boxtimes\dots\boxtimes \eta_k\\
			J_M(\beta,\mfa):=J(\beta,\mfa)\cap M=&J(\beta,\mfa^1)\times\dots\times J(\beta,\mfa^k),\quad \kappa_M:=\kappa\rest_{J_M(\beta,\mfa)}=\kappa_1\boxtimes\dots\boxtimes \kappa_k\\
			\bs{J}_{M}(\beta,\mfa):=\bs{J}(\beta,\mfa)\cap M=&\bs{J}(\beta,\mfa^1)\times\dots\times \bs{J}(\beta,\mfa^k)
		\end{aligned}
	\end{equation*}
	of groups and representations, where each $\theta_i$ is the transfer of $\theta$ as a simple character of $H^1(\beta,\mfa^i)$, and $\eta_i$ is the Heisenberg representation of $\theta_i$, and $\kappa_i$ is a $\beta$-extension of $\eta_i$. Let
	$$J_P^1(\beta,\mfa)=H^1(\beta,\mfa)(J^1(\beta,\mfa)\cap P)\quad\text{and}\quad J_P(\beta,\mfa)=H^1(\beta,\mfa)(J(\beta,\mfa)\cap P).$$ 
	Then we have $$J_P(\beta,\mfa)/J_P^1(\beta,\mfa)\cong J_M(\beta,\mfa)/J_M^1(\beta,\mfa)\cong\mcm.$$
	Let $\kappa_P$ be the representation of $J_P(\beta,\mfa)$ defined on the subspace of $J(\beta,\mfa)\cap N^-$-fixed vectors of $\eta$. Then we have $\kappa_P\rest_{J_M(\beta,\mfa)}=\kappa_M$ and $\mrind_{J_P(\beta,\mfa)}^{J(\beta,\mfa)}\kappa_P\cong \kappa$.
	
	If there is no ambiguity, we will omit $(\beta,\mfa)$ and simply write $H^1, J^1, J, \bs{J},H_M^1, J_M^1, J_M, \bs{J}_M,J_P^1,J_P$ for the above groups for short.
	
	In particular, in the case $m_1=m_2=\dots=m_k=m/k$, we have $V^1\cong\dots\cong V^k$, $\mfa^1\cong\cdots\cong\mfa^k$, $\mfb^1\cong\cdots\cong\mfb^k\cong\mrm_{m/k}(\mfo_E)$, $H^1(\beta,\mfa^1)\cong\dots\cong H^1(\beta,\mfa^k)$, $\theta_1\cong\dots\cong\theta_k$, $J^1(\beta,\mfa^1)\cong\dots\cong J^1(\beta,\mfa^k)$, $\eta_1\cong\dots\cong\eta_k$, $J(\beta,\mfa^1)\cong\dots\cong J(\beta,\mfa^k)$, $\kappa_1\cong\dots\cong\kappa_k$.
	
	\subsection{Simple types}\label{subsectionsimpletypes}
	
	Keep the notations of the previous part. Fix an $n$-fold cover $\ol{G}$ of $G=\mrgl_r(F)$ with respect to the $2$–cocycle $\sigma_{\mrdet}^{\bs{c}}\cdot\sigma_{\mathrm{KP}}^{\bs{d}}$ and a splitting $\bs{s}:K\rightarrow\ol{G}$, where $K$ is a maximal open compact subgroup of $G$. 
	
	Fix a simple stratum $[\mfa,u,0,\beta]$ in $A=\mrm_r(F)$, such that $U(\mfa)\subset K$. In particular, we assume that $\mfb=B\cap\mfa$ is isomorphic to the $\mfo_E$-algebra \eqref{eqmfb} with $m_1=\dots=m_k=m/k=:m_0$. 
	
	We consider the decomposition of $E$-vector spaces $V=\bigoplus_{i=1}^{k}V^{i}$, an $E$-basis $\{v_{i1},\dots,v_{im_{0}}\}$ of each $V^i$ and the related parabolic group $P=MN$ of $G$ as before.
	
	For our later use, we may construct another maximal simple stratum $[\amax,u_{\text{max}},0,\beta]$ in $A$, such that 
	\begin{itemize}
		\item $\amax$ is a hereditary order in $A$ normalized by $E^{\times}$, such that $U(\mfa)\subset U(\amax)\subset K$;
		\item under the basis $\{v_1',\dots, v_m'\}:=\{v_{11},\dots,v_{1m_{0}},\dots,v_{k1},\dots,v_{km_{0}}\}$ the $\mfo_E$-order $\bmax=B\cap \amax$ is identified with $\mrm_m(\mfo_E)$;
		\item $-u_{\text{max}}$ equals the valuation of $\beta$ with respect to $\amax$.
	\end{itemize}
	We consider the groups $\Honemax:=H^1(\beta,\amax)$, $\Jonemax:=J^1(\beta,\amax)$, $\Jmax:=J(\beta,\amax)$, $\bs{\Jmax}:=E^\times\Jmax$ and $J':=U(\mfb)\Jonemax$. Then we have
	$$\Honemax\cap M=H_M^1,\quad \Jonemax\cap M=J_M^1\quad\mathrm{and}\quad \Jmax\cap M=J_M$$
	and 
	$$\Jmax/\Jonemax\cong U(\bmax)/U^1(\bmax)\cong\mcg=\mrgl_{m}(\bs{l})\quad J'/\Jonemax\cong U(\mfb)/U^1(\bmax)\cong\mcp,$$
	where $\mcp$ is a parabolic subgroup of $\mcg$ whose Levi subgroup is $\mcm$.
	
	We explain the construction of a simple type $(\ol{J},\lambda)$ of $\ol{G}$ as follows.
	
	\begin{itemize}
		\item Let $\theta$ be a simple character of $H^1$, and $\eta$ a Heisenberg representation of $\theta$ and $\kappa$ a $\beta$-extension of $\eta$.
		
		\item We extend $\kappa$ to a (non-genuine) representation $1_{\mu_n}\cdot\kappa$ of $\ol{J}$.
		
		\item Let $\varrho_0$ be a cuspidal representation of $\mrgl_{m_0}({\bs{l}})$ and $\varrho=\varrho_0\boxtimes\dots\boxtimes\varrho_0$ a cuspidal representation of $\mcm\cong\underbrace{\mrgl_{m_0}({\bs{l}})\times\dots\times\mrgl_{m_0}({\bs{l}})}_{k\text{-copies}}$.
		
		\item Let $\rho=\mrinf_{\mcm}^{J}(\varrho)$ be the inflation of $\varrho$ as an representation of $J$, and $\epsilon\cdot\rho$ the genuine representation of $\ol{J}$ extending the representation $\rho$ of $\bs{s}(J)$.
		
		\item Finally, let $\lambda=(1_{\mu_n}\cdot\kappa)\otimes(\epsilon\cdot\rho)$ be a genuine representation of $\ol{J}$.
	\end{itemize}      Then $(\ol{J},\lambda)$ is a simple type of $\ol{G}$.
	
	We construct pairs $(\ol{J_P},\lambda_P)$ and $(\ol{J_M},\lambda_M)$ related to $(\ol{J},\lambda)$ as follows.
	
	\begin{itemize}
		\item We consider representations $\kappa_P$ and $\kappa_M$ as before.
		
		\item We extend $\kappa_P$ (resp. $\kappa_M$) to a (non-genuine) representation $1_{\mu_n}\cdot\kappa_P$ (resp. $1_{\mu_n}\cdot\kappa_M$) of $\ol{J_P}$ (resp. $\ol{J_M}$).
		
		\item Let $\rho_P=\mrinf_{\mcm}^{J_P}(\varrho)$ (resp. $\rho_M=\mrinf_{\mcm}^{J_M}(\varrho)$) be the inflation of $\varrho$ as an representation of $J_P$ (resp. $J_M$), and $\epsilon\cdot\rho_P$ (resp. $\epsilon\cdot\rho_M$) the genuine representation of $\ol{J_P}$ (resp.  $\ol{J_M}$) extending the representation $\rho_P$ (resp. $\rho_M$) of $\bs{s}(J_P)$ (resp. $\bs{s}(J_M)$).
		
		\item Finally, let $\lambda_P=(1_{\mu_n}\cdot\kappa_P)\otimes(\epsilon\cdot\rho_P)$ (resp. $\lambda_M=(1_{\mu_n}\cdot\kappa_M)\otimes(\epsilon\cdot\rho_M)$) be a genuine representation of $\ol{J_P}$ (resp.  $\ol{J_M}$).
	\end{itemize}   
	From our construction, we necessarily have $\lambda_P\rest_{\ol{J_M}}=\lambda_M$ and $\mrind_{\ol{J_P}}^{\ol{J}}\lambda_P\cong \lambda$.
	
	We construct another pair $(\ol{J'},\lambda')$ related to $(\ol{J},\lambda)$ as follows.
	
	\begin{itemize}
		\item We construct the simple character $\theta_{\text{max}}$ of $\Honemax$ as the transfer of $\theta$. In particular, we have $\theta_{\text{max}}\rest_{H^1_M}=\theta_M.$
		
		\item We construct the Heisenberg representation of $\eta_{\text{max}}$ of $\Jonemax$. In particular, we have 
		\begin{equation}\label{eqetathetamax}
			[\Jonemax:\Honemax]^{1/2}\cdot\eta_{\text{max}}\cong\mrind_{\Honemax}^{\Jonemax}(\theta_{\text{max}})
		\end{equation}
		
		\item  We also construct a unique $\beta$-extension $\kappa_{\text{max}}$ of $\Jmax$ extending $\eta_{\text{max}}$, such that 
		$$\mrind_{J'}^{U(\mfb)U^1(\mfa)}(\kappa_{\text{max}}\rest_{J'})\cong\mrind_{J}^{U(\mfb)U^1(\mfa)}(\kappa)\cong\mrind_{J_P}^{U(\mfb)U^1(\mfa)}(\kappa_P).$$ 
		
		\item We extend $\kappa_{\text{max}}$ to a (non-genuine) representation $1_{\mu_n}\cdot\kappa_{\text{max}}$ of $\ol{\Jmax}$.
		
		\item We extend $\varrho$ to a representation $\varrho'$ of $\mcp$ whose restriction to the unipotent radical is trivial. 
		
		\item Let $\rho'=\mrinf_{\mcp}^{J'}(\varrho')$ be the inflation of $\varrho'$ as a representation of $J'$, and $\epsilon\cdot\rho'$ the genuine representation of $\ol{J'}$ extending the representation $\rho'$ of $\bs{s}(J')$.
		
		\item Finally, let $\lambda'=(1_{\mu_n}\cdot\kappa_{\text{max}}\rest_{\ol{J'}})\otimes(\epsilon\cdot\rho')$ be a genuine representation of $\ol{J'}$.
	\end{itemize} 
	We necessarily have
	\begin{equation}\label{eqisolambdalanbda'}
		\mrind_{\ol{J'}}^{\ol{U(\mfb)}\ol{U^1(\mfa)}}(\lambda')\cong\mrind_{\ol{J}}^{\ol{U(\mfb)}\ol{U^1(\mfa)}}(\lambda)\cong\mrind_{\ol{J_P}}^{\ol{U(\mfb)}\ol{U^1(\mfa)}}(\lambda_P). 
	\end{equation}
	
	We may analyze the intertwining sets of $\lambda$, $\lambda_M$, $\lambda_P$. Identifying $B$ with $\mrm_m(E)$ via the given basis, we  consider the following subgroups of $B^\times$:
	\begin{equation*}
		\begin{aligned}
			W_0(\mfb)&=\{(\delta_{\sigma(i)j}I_{m_0})_{ 1\leq i,j\leq k}\mid \sigma\in\mfS_k \}\\
			T_B&=\{\mrdiag(b_1,b_2\dots, b_m)\mid b_1,b_2,\dots,b_m\in E^\times\}\\
			T(\mfb)&=\{\mrdiag(\varpi_E^{s_1}I_{m_0},\varpi_E^{s_2}I_{m_0}\dots, \varpi_E^{s_k}I_{m_0})\mid s_1,s_2,\dots,s_k\in \mbz\}\\
			T(\mfb,\varrho)&=\{\mrdiag(\varpi_E^{s_1}I_{m_0},\varpi_E^{s_2}I_{m_0}\dots, \varpi_E^{s_k}I_{m_0})\mid s_1,s_2,\dots,s_k\in \mbz, \\
			&\quad\quad l_0[(s_1+\dots+s_k)(2\bs{c}+\bs{d})r_0-s_i\bs{d}]\equiv 0\ (\text{mod}\ n),\ i=1,2,\dots,k\}\\
			W(\mfb)&=W_0(\mfb) T(\mfb)\\
			W(\mfb,\varrho)&=W_0(\mfb) T(\mfb,\varrho)
		\end{aligned}
	\end{equation*}
	where $(\delta_{ij})_{1\leq i,j\leq k}=I_k$, and $r_0=r/k$, and $l_0$ denotes the maximal positive integer $l$ dividing $n$ such that $\varrho_0$ is isomorphic to its twist by a character of order $l$. 
	
	In particular, $W(\mfb)$ and $W(\mfb,\varrho)$ are indeed groups since $W_0(\mfb)$ normalizes $T(\mfb)$ and $T(\mfb,\varrho)$. Then $T(\mfb)$ and $T(\mfb,\varrho)$ are $\mbz$-lattices of rank $k$ endowed with an $\mfS_k$-action (induced by the $W_0(\mfb)$-conjugation), given by $$\mrdiag(\varpi_E^{s_1}I_{m_0},\varpi_E^{s_2}I_{m_0}\dots, \varpi_E^{s_k}I_{m_0})\mapsto\mrdiag(\varpi_E^{s_{\sigma(1)}}I_{m_0},\varpi_E^{s_{\sigma(2)}}I_{m_0}\dots, \varpi_E^{s_{\sigma(k)}}I_{m_0})$$ for  $\sigma\in\mfS_k$. Let $$\mcx(\lambda)=T(\mfb)/T(\mfb,\varrho)$$
	which is a finite abelian group endowed with an induced $\mfS_k$-action.
	
	Let $n_0=n/\mrgcd(n,(2\bs{c}+\bs{d})r_0l_0,\bs{d}l_0)$, and $\sigma_0$ the permutation $1\mapsto 2\mapsto \dots\mapsto k\mapsto 1$, and $$\Pi_E=\mrdiag(\underbrace{I_{m_0},\dots,I_{m_0},\varpi_{E}^{n_0}I_{m_0}}_{k\text{-terms}})\cdot(\delta_{\sigma_0(i)j}I_{m_0})_{ 1\leq i,j\leq k}\in B.$$
	Let $d_0=n/\mrgcd(n,l_0(2\bs{c}r+\bs{d}r-\bs{d}))$, and
	$$\zeta_E=\mrdiag(\underbrace{\varpi_{E}^{d_0}I_{m_0},\dots,\varpi_{E}^{d_0}I_{m_0},\varpi_{E}^{d_0}I_{m_0}}_{k\text{-terms}})\in B$$
	In particular, $d_0$ divides $n_0$. We may verify directly that both $\Pi_E$ and $\zeta_E$ lies in $W(\mfb,\varrho)$. Indeed, let 	
	$$T_0(\mfb)=\{\mrdiag(\varpi_E^{n_0s_1}I_{m_0},\varpi_E^{n_0s_2}I_{m_0}\dots, \varpi_E^{n_0s_k}I_{m_0})\mid s_1,s_2,\dots,s_k\in \mbz\}$$
	be defined as a subgroup of $T(\mfb,\varrho)$ of finite index. Then the subgroup of $W(\mfb,\varrho)$ generated by  $W_0(\mfb)$ and $\Pi_E$ is $W_0(\mfb)T_0(\mfb)$.

	Let $W'(\mfb,\varrho)$ be the subgroup of $W(\mfb,\varrho)$ of finite index generated by $W_0(\mfb)$, $\Pi_E$ and $\zeta_E$. 
	
	Indeed, the groups $W(\mfb)$, $W(\mfb,\varrho)$ and $W'(\mfb,\varrho)$ are (extended) affine Weyl groups of type $A$.
	
	\begin{proposition}[\cite{zou2023simple}*{\S 6.2, \S 6.3}]\label{propintertwinigset}
		
		\begin{enumerate}
			\item We have $I_G(\lambda)=JW(\mfb,\varrho)J$, $I_G(\lambda_P)=J_PW(\mfb,\varrho)J_P$ and $I_M(\lambda_M)=N_M(\lambda_M)=T(\mfb,\varrho)J_M\subset\bs{J}_{M}=T(\mfb)J_M$. 
			
			\item Moreover, the related intertwining spaces are of multiplicity one.
			
			\item In the case of a Kazhdan--Patterson cover or the Savin cover, we further have $W(\mfb,\varrho)=W'(\mfb,\varrho)$. More precisely, for the former case we have $T(\mfb,\varrho)=\pairangone{\zeta_E}T_0(\mfb)$; for the latter case we have $T(\mfb,\varrho)=T_0(\mfb)$.
			
		\end{enumerate}
		
	\end{proposition}
	
	As a result, we have $\mcx(\lambda)\cong\bs{J}_M/N_M(\lambda_M)$.
	
	Indeed, $(\ol{J_M},\lambda_M)$ is a maximal simple type of $\ol{M}$ in the following sense:
	
	\begin{proposition}[\cite{zou2023simple}*{Theorem 6.15}]
		
		\begin{enumerate}
			\item We may extend $\lambda_M$ to an irreducible representation $\lambda_{M}'$ of $N_{\ol{M}}(\lambda_M)$.
			\item The induction $\bs{\lambda}_{M}=\mrind_{N_{\ol{M}}(\lambda_M)}^{\ol{\bs{J}_{M}}}\lambda_{M}'$ is an irreducible representation of $\ol{\bs{J}_{M}}$.
			
			\item The compact induction $\pi_{M}=\mrind_{\ol{\bs{J}_{M}}}^{\ol{M}}\bs{\lambda}_{M}$ is an irreducible cuspidal representation of $\ol{M}$. 
			
			\item Moreover, an irreducible representation $\pi_{M}'$ of $\ol{M}$ is in the inertial equivalence class of $\pi_{M}$ if and only if its restriction to $\ol{J_{M}}$ contains $\lambda_{M}$.
		\end{enumerate} 
		
	\end{proposition}
	
	\begin{example}
		
		We consider depth 0 simple types. Up to $G$-conjugacy we have $E=F$, $r=m=m_0k$, $K=\mrgl_{r}(\mfo_{F})$,    $\mfa=\{(a_{ij})_{1\leq i,j\leq k}\mid a_{ij}\in \mrm_{m_0}(\mfo_F),\ i\leq j\ \text{and} \ a_{ij}\in \mrm_{ m_0}(\mfp_F),\ i>j\}, $ and $\theta=\eta$ is the trivial character of $H^1=J^1=U^1(\mfa)$, and $\kappa$ is the trivial character of $J=U(\mfa)$. Let $\varrho=\varrho_0\boxtimes\dots\boxtimes\varrho_0$ be a cuspidal representation of $\mcm\cong U(\mfa)/U^1(\mfa)\cong\mrgl_{m_0}({\bs{l}})\times\dots\times\mrgl_{m_0}({\bs{l}})$ with $\varrho_0$ being a cuspidal representation of $\mrgl_{m_0}({\bs{l}})$. Then, $P$ is the upper block triangular parabolic subgroup with respect to the composition $m=m_0+\dots+m_0$ and we have $J=J_P=J'$, $\rho=\rho_P=\rho'=\mrinf_{\mcm}^{J}(\varrho)$ and $\lambda=\lambda_P=\lambda'=\epsilon\cdot\rho$. 
		
	\end{example}
	
	\subsection{Hecke algebras} In this part, we recall the known results on the Hecke algebra $\mch(\ol{G},\lambda)$, $\mch(\ol{G},\lambda_P)$, $\mch(\ol{G},\lambda')$ and $\mch(\ol{M},\lambda_M)$. 
	
	Recall that for a general $\ell$-group $G$, its open compact subgroup $H$ and an irreducible representation $(\rho,W)$ of $H$, the Hecke algebra $\mch(G,\rho)$ is defined as the convolution algebra of bi-$\rho^{\vee}$-equivariant compactly supported functions from $G$ to the endomorphism space  $\mrend_\mbc(\rho^{\vee})$, or equivalently the opposite of the endomorphism algebra $\mrend_G(\mrind_H^G(\rho))$. Here, $\rho^{\vee}$ denotes the contragredient of $\rho$.
	
	In particular, $\mch(G,\rho)$ gives a right action on the induced representation $\mrind_H^G(\rho)$. This gives a functor
	\begin{equation}\label{eqMrho}
		\bs{\mathrm{M}}_\rho:\mrrep(G)\rightarrow\mrmod(\mch(G,\rho)),\quad \pi\mapsto\mrhom_G(\mrind_H^G(\rho),\pi)\cong\mrhom_H(\rho,\pi\rest_H),
	\end{equation}
	where $\mrmod(\mch(G,\rho))$ denotes the category of left $\mch(G,\rho)$-modules. More precisely, for $\phi:G\rightarrow\mrend(\rho^{\vee})$, we define $$ \phi^{\vee}\in\mch(G,\rho^{\vee}),\quad \phi^{\vee}(g)=\phi(g^{-1})^{\vee},$$ 
	where $ ^{\vee}:\mrend_\mbc(\rho^\vee)\rightarrow\mrend_\mbc(\rho)$ is the usual dual map. Then,  the action of $\phi$ on $\bs{\mathrm{M}}_\rho(\pi)$ is given by the convolution product:
	$$(\phi\ast\varphi):v\mapsto\int_{g\in G}\pi(g)\varphi(\phi^\vee(g)v),\quad v\in W,\ \varphi\in\mrhom_H(\rho,\pi\rest_H).$$
	
	\begin{theorem}[\cite{zou2023simple}*{\S 3.4, Theorem 6.15, Theorem 6.16}, \S 7.4]\label{thmheckealg}
		
		\begin{enumerate}
			\item We have a natural isomorphism of algebras
			$$\mch:=\mch(\ol{G},\lambda)\cong\mch(\ol{G},\lambda_P)\cong\mch(\ol{G},\lambda'),$$
			whose supports are given by the intertwining sets in Proposition \ref{propintertwinigset}. In particular, $$\mch(\ol{\Jmax},\lambda')\cong\mch(\ol{U(\bmax)},\epsilon\cdot\rho')\cong \mch(U(\bmax),\rho')\cong\mch(\mcg,\varrho')$$ is a subalgebra of $\mch(\ol{G},\lambda')$ isomorphic to the finite Hecke algebra $\mch_0:=\mch(\mfS_k,\bs{q}_0)$ of type A in the sense of \cite{solleveld2021affine}*{\S 1.1},
			where $\bs{q}_0=q^{[\bs{l}:\bs{k}]m_0}$. In this way, we identify $\mch_{0}$ with a subalgebra of $\mch$.
			
			\item We may define an isomorphism between $\mch(\ol{M},\lambda_M)$ and the commutative algebra $\mca:=\mbc[T(\mfb,\varrho)]$ that is $\mfS_{k}$-equivariant, such that preimage of each $t\in T(\mfb,\varrho)$ is a function $\phi_t:\ol{G}\rightarrow\mrend_\mbc(\lambda_M^{\vee})$ supported on $t\ol{J_M}$. 
			
			\item We may define a canonical algebra embedding $t_P:\mch(\ol{M},\lambda_M)\rightarrow \mch(\ol{G},\lambda_P)$  that is unique up to a scalar. It induces the pull-back map $$t_P^{\ast}:\mrmod(\mch(\ol{G},\lambda_P))\rightarrow\mrmod(\mch(\ol{M},\lambda_{M})).$$
			Thus combining (2) and (3), we identify $\mca$ with a subalgebra of $\mch$.
			\item There is a cuspidal inertial equivalence class $\mfs_M=(\ol{M},\mco)$ of $\ol{M}$, such that $(\ol{J_M},\lambda_M)$ is a type related to $\mfs_M$. Here, $\mco$ is an orbit of a genuine cuspidal representation of $\ol{M}$ under the action of unramified characters of $M$. Then, we have an equivalence of categories
			$$\bs{\mathrm{M}}_{\lambda_M}:\mrrep_{\mfs_M}(\ol{M})\rightarrow\mrmod(\mch(\ol{M},\lambda_M)),$$
			where $\mrrep_{\mfs_M}(\ol{M})$ denotes the subcategory of $\mrrep(\ol{M})$ of smooth representations of inertial equivalence class $\mfs_M$.
			\item $(\ol{J_P},\lambda_P)$ is a covering pair of $(\ol{J_M},\lambda_M)$ and $(\ol{J_P},\lambda_P)$ is a type related to the inertial equivalence class $\mfs_G=(\ol{G},\mco)$ of $\ol{G}$. Then, we have an equivalence of categories
			$$\bs{\mathrm{M}}_{\lambda_P}:\mrrep_{\mfs_G}(\ol{G})\rightarrow\mrmod(\mch(\ol{G},\lambda_P)),$$
			where $\mrrep_{\mfs_G}(\ol{G})$ denotes the subcategory of $\mrrep(\ol{G})$ of smooth representations of inertial equivalence class $\mfs_G$.
			\item Let $r_N:\mrrep(\ol{G})\rightarrow \mrrep(\ol{M})$ be the normalized Jacquet functor with respect to the parabolic group $\ol{P}=\ol{M}N$, then we have the following commutatative diagram:
			$$\xymatrix{
				\mrrep_{\mfs_{G}}(\ol{G}) \ar[d]_-{r_{N}} \ar[r]^-{\bs{\mathrm{M}}_{\lambda_P}} & \mrmod(\mch(\ol{G},\lambda_P)) \ar[d]^-{t_{P}^{*}} \\ 
				\mrrep_{\mfs_{M}}(\ol{M}) \ar[r]_-{\bs{\mathrm{M}}_{\lambda_{M}}} & \mrmod(\mch(\ol{M},\lambda_{M}))
			}$$ 
		\end{enumerate}	
		
	\end{theorem}
	
	To continue, we also briefly recall the concept of an affine Hecke algebra, and more precisely its Bernstein presentation, following \cite{solleveld2021affine}*{Section 1}.
	
	Let $(W_0,S_0)$ be a finite Coxeter system with $W_0$ the finite Weyl group and $S_0$ a set of generators, and $(R,\Delta)$ the related root system with $R$ the set of roots and $\Delta$ a set of simple roots. To simplify our discussion, we assume that $(R,\Delta)$ is of type A. Let $\bs{q}_0>1$ be a parameter. For $s,s'\in S_0$, let $m(s,s')$ be the order of $ss'$ in $W_0$. Then, the finite Hecke algebra $\mch(W_0,\bs{q}_0)$ is a $\mbc$-algebra with vector basis $\{T_w\mid w\in W_0\}$, subjected to the relations
	\begin{itemize}
		\item $T_w=\prod_{i=1}^{l(w)}T_{s_i}$, where $w=\prod_{i=1}^{l(w)}s_i$ is any reduced product of $w$ with $s_i\in S_0$;
		\item $(T_s+1)(T_s-\bs{q}_0)=0$, $s\in S_0$;
		\item $\underbrace{T_sT_{s'}T_s\dots}_{m(s,s')\text{-terms}}=\underbrace{T_{s'}T_sT_{s'}\dots}_{m(s,s')\text{-terms}}$.
	\end{itemize}
	Let $Y$ be a cocharacter lattice endowed with the $W_0$-action, such that $\mcr=(Y^{\vee},R,Y,R^{\vee},\Delta)$ is a based root datum.  Let $\{\phi_t\mid t\in Y\}$ be the standard basis of the $\mbc$-algebra $\mbc[Y]$.
	
	We define the affine Hecke algebra $\mch(\mcr,\bs{q}_0)$, whose underlying vector space is $\mbc[Y]\otimes_\mbc\mch(W_0,\bs{q}_0)$ with the subjected rules:
	\begin{itemize}
		\item $\mbc[Y]$ and $\mch(W_0,\bs{q}_0)$ are embedded as subalgebras;
		\item For $\alpha\in\Delta$ and $t\in Y$, we have the Bernstein presentation (in the type A case):
		\begin{equation}\label{eqbernstein}
			\phi_t\ast T_{s_\alpha}- T_{s_\alpha}\ast\phi_{s_\alpha\cdot t}=(\bs{q}_0-1)\frac{\phi_t-\phi_{s_\alpha\cdot t}}{\phi_{0}-\phi_{-\alpha^{\vee}}}.
		\end{equation}
	\end{itemize}
	
	Come back to our previous discussion. By Theorem \ref{thmheckealg}, we may regard both $\mca$ and $\mch_0$ as subalgebras of $\mch$. Notice that $\mca\otimes_{\mbc}\mch_{0}$ has an affine Hecke algebra structure with $W_0=\mfS_k$, $Y=T(\mfb,\varrho)$ and $\bs{q}_0=q^{[\bs{l}:\bs{k}]m_0}$, where $\mca$ is the related ``lattice part'' and $\mch_0$ is the related ``finite'' part. However, a priori it is not clear if the above algebra is isomorphic to $\mch$ or not.
	
	So we need the following important assumption to proceed.
	
	\begin{assumption}[\cite{zou2023simple}*{Conjecture 7.16}]\label{assumHeckealg}
		
		Regarding $\mca$ and $\mch_0$ as subalgebras of $\mch$, then $\mch$ is  isomorphic to the affine Hecke algebra $\mca\otimes_{\mbc}\mch_0$.
		
	\end{assumption} 
	
	In the case of a Kazhdan--Patterson cover or the Savin cover the assumption is always satisfied (\cite{zou2023simple}*{Corollary 7.12}).
	
	In general, we don't know if the about assumption is true or not. However, if we consider the subalgebra $\mca_0:=\mbc[T_0(\mfb)]\subset\mca$. Then, the related affine Hecke algebra $\mca_0\otimes_{\mbc}\mch_0$ is embedded in $\mch$ as a subalgebra of finite index (\emph{cf.} \cite{zou2023simple}*{Theorem 7.5}).

	\section{Main theorem}
	
	In this part, we state the main theorem of this article.
	
	Let $G=\mrgl_r(F)$ and $\ol{G}$ an $n$-fold cover of $G$ as before. Let $B=TU$ be a Borel subgroup of $G$, where $T$ is a maximal split torus and $U$ is the related unipotent radical of $G$. 
	
	Let $\psi$ be a generic character of $U$, meaning that the restriction of $\psi$ to any rank one unipotent subgroup $U_{\alpha}$ related to a simple root $\alpha\in\Delta(G,B)$ is non-trivial. It is known that the pair $(U,\psi)$, called a Whittaker pair, is unique up to $G$-conjugacy. Using the canonical embedding $U\hookrightarrow\ol{G}$, we regard $U$ as a subgroup of $\ol{G}$.
	
	We consider the Gelfand--Graev representation $$\mcv:=\mrind_U^{\ol{G}}(\psi)$$
	of $\ol{G}$. Elements in $\mcv$ are regarded as compact supported complex valued functions on $\ol{G}$ that are left $(U,\psi)$-equivariant.
	
	Let $(\ol{J},\lambda)$ be a simple type of $\ol{G}$ and $\mch=\mch(\ol{G},\lambda)$, with the same notation in Section \ref{sectionsimpletype}. Define $$\mcv^{\lambda}:=\bs{\mathrm{M}}_\lambda(\mcv)=\mrhom_{\ol{G}}(\mrind_{\ol{J}}^{\ol{G}}(\lambda),\mrind_U^{\ol{G}}(\psi))$$
	
	Recall that we have an $\mfS_k$-action on $\mcx(\lambda)$. Given an $\mfS_k$-orbit $\mco$ in $\mcx(\lambda)$, we define
	\begin{itemize}
		\item $W_{\mco}$ the stabilizer of $\mco$ in $\mfS_k$;
		\item $\mch_{\mco}=\mch(W_{\mco},\bs{q}_0)$ the related finite Hecke algebra as a subalgebra of $\mch_0$;
		\item $\varepsilon_{\mco}$ the sign character of $\mch_{\mco}$;
		\item $\mch_0\otimes_{\mch_{\mco}}\varepsilon_{\mco}=\mrInd_{\mch_{\mco}}^{\mch_{0}}(\varepsilon_{\mco})$ the related induced representation of $\mch_{0}$.
	\end{itemize}
	Our goal is to describe the decomposition of $\mcv^{\lambda}$ as an $\mch$-module.
	\begin{theorem}\label{thmmain}
		
		If $\ol{G}$ is either a Kazhdan--Patterson cover or the Savin cover, then we have a decomposition of $\mch$-modules $$\mcv^{\lambda}=\bigoplus_{\mco\in\mcx(\lambda)/\mfS_k }\mcv_{\mco}^{\lambda},$$
		such that each $\mcv_{\mco}^{\lambda}$ is isomorphic to  $\mca\otimes_{\mbc}(\mch_{0}\otimes_{\mch_{\mco}}\varepsilon_{\mco})$.
		
	\end{theorem}
	
	\begin{remark}
		
		When $n=1$, we have $\mcx(\lambda)=\{1\}$. In this case, we have $\mcv^{\lambda}\cong\mca\otimes_{\mbc}(\mch_{0}\otimes_{\mch_{0}}\varepsilon_{0})$, which recovers \cite{chan2019bernstein}*{Theorem 3.4}, where $\varepsilon_0$ is the sign character of $\mch_{0}$.
		
	\end{remark}
	
	\begin{remark}
		
		For the depth zero simple type $(J,\lambda)$ where $J=I$ is an Iwahori subgroup of $G$ and $\lambda=\epsilon\cdot1_{I}$, our result recovers \cite{gao2024genuine}*{Theorem 1.2} for Kazhdan--Patterson covers and Savin's cover. In particular, for such covers every orbit $\mco\in\mcx(\lambda)/\mfS_k$ is splitting in the sense of \emph{loc. cit.}.
		
	\end{remark}
	
	\section{Whittaker datum}
	
	We fix a simple type $(\ol{J},\lambda)$ of $\ol{G}$ as above. As we have already noticed, the isomorphism class $\mcv=\mrind_U^{\ol{G}}(\psi)$ depends only on the $G$-conjugacy class of the Whittaker datum $(U,\psi)$. 
	
	The goal of this section is to seek for a suitable Whittaker datum $(U,\psi)$ relative to $(\ol{J},\lambda)$, such that $\mcv^\lambda$ is easily described.
	
	\subsection{Result of Paskunas--Stevens}\label{subsectionPSresult}
	
	We recall the result in \cite{paskunas2008realization} to construct a suitable Whittaker datum. Let $[\mfa,u,0,\beta]$ be the simple stratum in $A=\mrend_F(V)$ related to $(\ol{J},\lambda)$. We also have the related maximal simple stratum $[\amax,u_{\text{max}},0,\beta]$ in $A$. Fix additive characters $\psi_F:F\rightarrow\mbc^\times$ and $\psi_E:E=F[\beta]\rightarrow\mbc^\times$ of conductor $\mfp_F$ and $\mfp_E$ respectively, such that $\psi_E\rest_F=\psi_F$. 
	
	Recall that we have chosen an $E$-basis 
	$\{v_1',\dots, v_m'\}$ of $V$, under which $B=\mrend_E(V)$ is identified with $\mrm_m(E)$, and $\mfb=B\cap \mfa$ is identified with the standard hereditary order in $B$ with respect to the composition $m=m_0+m_0+\dots+m_0$, and $\bmax=B\cap \amax$ is identified with $\mrm_m(\mfo_E)$.
	
	Consider the $F$-flags $$\mcf_0:\{0\}\subset Ev_1'\subset Ev_1'\oplus Ev_2'\subset\dots\subset Ev_1'\oplus\dots\oplus Ev_m'=V$$
	and
	$$\mcf_0^{-}:\{0\}\subset Ev_m'\subset Ev_{m-1}'\oplus Ev_m'\subset\dots\subset Ev_1'\oplus\dots\oplus Ev_m'=V$$
	of $V$. Let $P_0=M_0N_0$ be the related parabolic subgroup of $G$ that fixes $\mcf_0$, where $M_0$ is a related Levi factor and $N_0$ the unipotent radical. Similarly, $P_0^-=M_0N_0^-$ is the related opposite parabolic subgroup of $G$ that fixes $\mcf_0^{-}$. In particular, $B^\times \cap P_0$ (resp. $B^\times \cap P_0^-$) is identified with the upper (resp. lower) triangular Borel subgroup of $B^\times\cong\mrgl_m(E)$.

	Since $$\mcf':\{0\}\subset V^1\subset V^1\oplus V^2\subset\dots\subset  V^1\oplus\dots\oplus V^k=V$$ is an $F$-flag contained in $\mcf_0$ (meaning that all the intermediate subspaces of $\mcf'$ are intermediate spaces in $\mcf_0$), we have $P_0\subset P$, $M_0\subset M$ and $N\subset N_0$. 
	
	Similarly, given a full $F$-flag $$\mcf:\{0\}=V_0\subset V_1\subset V_2\subset\dots \subset V_r=V$$ of $V$ with $\mrdim_F(V_i)=i$, we may consider the related Borel subgroup $TU$ that fixes $\mcf$, where $T$ is the related maximal torus and $U$ the related unipotent radical. Define $$X_\mcf=\{x\in A\mid xV_i\subset V_{i+1},\ i=0,1,\dots,r-1\}.$$
	Given an element $x\in X_\mcf$, the map $u\mapsto\psi_x(u):=\psi_F\circ\mrtr_{A/F}(x(u-1))$ defines a character of $U$. 
	
	The following theorem is due to Paskunas--Stevens. 
	
	\begin{theorem}[\cite{paskunas2008realization}*{Theorem 2.6, Theorem 3.3, Corollary 3.4, Corollary 3.5}]\label{thmsimpletypewhittaker}
		
		\begin{enumerate}
			
			\item There exist a full $F$-flag $\mcf$ as above containing $\mcf_0^-$ 
			and a character $\chi$ of $U$, such that
			$\chi$ is trivial on $N_0^-$ 
			and $\chi\rest_{ \Honemax\cap U}=\theta_{\text{max}}\rest_{ \Honemax\cap U}$. In particular, we have $U_{B}:=B^\times \cap N_0^-=B^\times \cap U$, which is a unipotent subgroup of $B^{\times}$. 
			
			\item Let $\vartheta$ be the character of $(\Jmax\cap U)\Honemax$ extending $\chi$ and $\theta_{\text{max}}$, then for any $\beta$-extension we have $$\kappa_{\text{max}}\rest_{(\Jmax\cap U)\Jonemax}\cong\mrind_{(\Jmax\cap U)\Honemax}^{(\Jmax\cap U)\Jonemax}(\vartheta).$$
			
			\item There exists an element $b\in \amax\cap X_\mcf$, such that 
			$$\ul{\psi}_b:\mcu:=(U(\bmax)\cap U)/(U^1(\bmax)\cap U)\rightarrow\mbc^\times,\quad u(U^1(\bmax)\cap U)\mapsto\psi_b(u)$$
			defines a generic character of the unipotent radical $\mcu$ of $\mcg=U(\bmax)/U^1(\bmax)$. Moreover, $\psi_{b}$ is trivial on $U\cap M_0$ \footnote{This part follows from the proof of \cite{paskunas2008realization}*{Theorem 3.3}}.
			
			\item $\psi:=\chi\psi_b$ is a generic character of $U$ whose restriction to $U_B$ is also generic.
			
			\item Up to a suitable choice of $\mcf, b,\chi$ as above, we may further find an $F$-basis $\{e_1,\dots,e_d\}$ of $E$ and an $F$-basis $\{v_1,\dots,v_r\}$ of $V$ with 
			$$v_{d(i-1)+j}=e_d^{i-1}e_jv_{m+1-i}',\quad 
			1\leq i\leq m,\ 1\leq j\leq d,$$
			to identify $A$ with $\mrm_n(F)$, such that 
			$$\psi(u)=\psi_F(\sum_{i=1}^{r-1}u_{i+1,i})$$
			if $u\in U$ is identified with the matrix $(u_{ij})$ in $\mrgl_r(F)$, and 
			$$\psi(u')=\psi_E(\sum_{i=1}^{m-1}u_{i+1,i}')$$
			if $u'\in U_B$ is identified with the matrix $(u_{ij}')$ in $\mrgl_m(E)$.
			
		\end{enumerate}
		
	\end{theorem}
	
	Remark that in the theorem we consider $\mcf_0^-$ instead of $\mcf_0$, which is slightly different from that in \emph{loc. cit.}. From now on, we let $(U,\psi)$ be the Whittaker datum chosen in the theorem, such that the flag $\mcf$ contains $\mcf_0^-$. The advantage of such choice shall be seen in the next subsection.
	
	\subsection{Jacquet module}\label{subsectionKJacofGG}
	
	We consider the Levi subgroup $M$ defined as before, and $U_M=M\cap U$ its unipotent radical. Then $\psi_M:=\psi\rest_{U_M}$ is a generic character, and $(U_M,\psi_M)$ is a related Whittaker pair of $M$. 
	
	We consider the related Gelfand--Graev representation $\mcv_M=\mrind_{U_M}^{\ol{M}}(\psi_M)$ of $\ol{M}$.
	
	\begin{proposition}[\cite{bushnell1998supercuspidal}*{Theorem 2.2},  \cite{banks1998heredity}*{Theorem 2}, \cite{kaplan2022note}*{Theorem 3.9}]\label{propGGmoduleJacquet}
		
		Assume that $N^-$ is contained in $U$.
		
		\begin{enumerate}
			\item The Jacquet module $r_N(\mcv)$ is isomorphic to $\mcv_M$.
			\item Moreover, let $N^0$ be an open compact subgroup of $N$ and $f\in \mcv$ a function supported on $U\ol{M}N^0$, such that $f(gn)=f(g)$ for any $g\in\ol{G},\ n\in N^0$. Then $$r_N(f)(m)=\mu(N^0)\delta_N^{1/2}(m)f(m),\quad m\in\ol{M},$$ where $\mu(N^0)$ is the volume of $N^0$ with respect to a suitable Haar measure and $\delta_N$ is the modulus character of $N$.
		\end{enumerate}		
		
	\end{proposition}
	
	Remark that since $\mcf$ contains $\mcf_0^-$, we have $N^{-}\subset N_0^{-}\subset U$, which verifies the assumption of the proposition in our case.
	
	\section{$\mca$-module structure of $\mcv_M^{\lambda_M}$ ---The cuspidal case}\label{sectioncuspidal}
	
	In this section, we study $\mcv_M^{\lambda_M}$ as an  $\mch(\ol{M},\lambda_M)$-module. Recall that $(\ol{J_M},\lambda_M)$ is a maximal simple type of $\ol{M}$, thus relates to a cuspidal inertial equivalence class of $\ol{M}$. In this case, we have $\mch(\ol{M},\lambda_M)\cong\mca$ via Theorem \ref{thmheckealg}.
	
	Our goal is to study the space
	$$\mcv_M^{\lambda_M}=\mrhom_{\ol{M}}(\mrind_{\ol{J_M}}^{\ol{M}}(\lambda_M),\mrind_{U_M}^{\ol{M}}(\psi_M)).$$
	We consider the canonical embedding $\lambda_M\hookrightarrow \mrind_{\ol{J_M}}^{\ol{M}}(\lambda_M)\rest_{\ol{J_M}}$ of $\ol{J_M}$-representations. In this way, $\lambda_M$ is identified with the subspace of functions in $\mrind_{\ol{J_M}}^{\ol{M}}(\lambda_M)$ that are supported on $\ol{J_M}$. 
	
	For $\varphi\in \mrhom_{\ol{M}}(\mrind_{\ol{J_M}}^{\ol{M}}(\lambda_M),\mrind_{U_M}^{\ol{M}}(\psi_M))$, we denote by $\varphi\rest_{\lambda_M}$ its restriction to the subspace $\lambda_M$. By the Frobenius reciprocity, it is identified with an element in 
	$$\mrhom_{\ol{J_M}}(\lambda_M,\mrind_{U_M}^{\ol{M}}(\psi_M)\rest_{\ol{J_M}})$$
	We define 
	\begin{equation}\label{eqsuppphi}
		\mrsupp(\varphi\rest_{\lambda_M})=\bigcup_{v\in \lambda_M}\mrsupp(\varphi(v)),
	\end{equation} 
	where $\varphi(v)\in\mrind_{U_M}^{\ol{M}}(\psi_M)\rest_{\ol{J_M}}$ and $\mrsupp(\varphi(v))$ denotes the related support of $\varphi(v)$ being regarded as the union of finitely many $U_M$-$\ol{J_M}$ double cosets.
	
	\begin{lemma}\label{lemmaVlambdasupport}
		
		\begin{enumerate}
			\item We have	$\mrhom_{J_M\cap U_M}(\kappa_M,\chi)\cong \mbc.$
			\item We have \begin{equation*}
				\begin{aligned}
					\mrhom_{ \ol{J_M^g}\cap U_M}(\lambda_M^g,\psi_M)\cong\begin{cases}
						\mbc&\quad\text{if}\ g\in\bs{J}_M,\\
						0&\quad\text{if}\ g\notin\bs{J}_M.
					\end{cases}
				\end{aligned}
			\end{equation*}
			\item Given $\varphi\in \mcv_M^{\lambda_M}$, we have $\mrsupp(\varphi\rest_{\lambda_M})\subset U_M\ol{\bs{J}_M}$.
			
			\item Given $g\in \bs{J}_M$, there exists $0\neq \varphi_g\in \mcv_M^{\lambda_M}$, unique up to a scalar, such that  $\mrsupp(\varphi_g\rest_{\lambda_M})= U_Mg\ol{J_M}$.
			
		\end{enumerate}
		
	\end{lemma}
	
	\begin{proof}
		
		First we prove (1). Recall that $\chi\rest_{H_M^1\cap U_M}=\theta_{\text{max}}\rest_{ \Honemax\cap U_M}=\theta_{M}\rest_{H_M^1\cap U_M}$. Let $\vartheta_M$ be the character of $(J_M\cap U_M)H_M^1$ extending $\chi\rest_{H_M^1\cap U_M}$ and $\theta_M$, then we have
		$$\kappa_{M}\rest_{(J_M\cap U_M)J_M^1}\cong\mrind_{(J_M\cap U_M)H_M^1}^{(J_M\cap U_M)J_M^1}(\vartheta_M).$$
		Indeed, it is deduced from Theorem \ref{thmsimpletypewhittaker} or \cite{paskunas2008realization}*{Theorem 2.6.(2)} by considering a $\beta$-extension of a maximal simple character of $M$ (instead of $G$), which is possible since we may use the original result and take the tensor product. Using the Frobenius reciprocity and Mackey theorem, we have
		\begin{equation*}
			\begin{aligned}
				\mrhom_{J_M\cap U_M}(\kappa_M,\chi)&\cong\mrhom_{J_M\cap U_M}(\mrind_{(J_M\cap U_M)H_M^1}^{(J_M\cap U_M)J_M^1}(\vartheta_M)\rest_{J_M\cap U_M},\chi)\\
				&\cong\bigoplus_{g\in(J_M\cap U_M)H_M^1\backslash(J_M\cap U_M)J_M^1/(J_M\cap U_M)}\mrhom_{J_M\cap U_M\cap J_M^g\cap U_M^g }(\vartheta_M^g,\chi)\\
				&\cong\bigoplus_{g\in(J_M\cap U_M)H_M^1\backslash(J_M\cap U_M)J_M^1/(J_M\cap U_M)}\mrhom_{J_M\cap U_M\cap U_M^g }(\chi^g,\chi)
			\end{aligned}
		\end{equation*}
		Taking $g=1$ in the direct sum, we have that $\mrhom_{J_M\cap U_M}(\kappa_M,\chi)\neq 0$.
		
		To show the multiplicity one result, we use the uniqueness of the Whittaker model of an irreducible (cuspidal) representation of $M$. More precisely, let $(J_M,\lambda_M^0)$ be a maximal simple type of $M$ such that $\lambda_M=\epsilon\cdot\lambda_M^0$, where we identify $\lambda_M^0$ with a representation of $\bs{s}(J_M)$. Indeed, for $\rho_M^0=\mrinf_{\mcm}^{J_M}(\varrho)$ we have $\lambda_M^0=\kappa_M\otimes\rho_M^0$. Let $\bs{\lambda}_M^0$ be an irreducible representation of $\bs{J}_M$ extending $\lambda_M^0$, then the compact induction $\mrind_{\bs{J}_M}^{M}(\bs{\lambda}_M^0)$ is a cuspidal representation of $M$. Using the uniqueness of Whittaker model, the Frobenius reciprocity and Mackey theory, we have
		\begin{equation}\label{eqlambdaMwhittaker}
			\begin{aligned}
				\mbc\cong\mrhom_{U_M}(\mrind_{\bs{J}_M}^{M}(\bs{\lambda}_M^0),\psi_M)&\cong\prod_{g\in \bs{J}_M\backslash M/U_M}\mrhom_{\bs{J}_M^g\cap U_M }((\bs{\lambda}_M^0)^g,\psi_M)\\
				&\cong\prod_{g\in \bs{J}_M\backslash M/U_M}\mrhom_{J_M^g\cap U_M}((\lambda_M^0)^g,\psi_M).		
			\end{aligned}
		\end{equation}
		The second isomorphism follows from the fact that $J_M\cap U_M= \bs{J}_M\cap U_M$.
		When $g\in \bs{J}_M$, we have $(\lambda_M^0)^g\cong\lambda_M^0$. Moreover, 
		$$\mrhom_{J_M\cap U_M}(\lambda_M^0,\psi_M)\cong\mrhom_{J_M\cap U_M}(\kappa_M,\chi)\otimes_\mbc\mrhom_{J_M\cap U_M}(\rho_M^0,\psi_b)$$
		is non-zero. It is because in the tensor product the first term is non-zero, and the second term is isomorphic to $\mrhom_{\mcu_M}(\varrho,\ul{\psi}_b)$, which is one-dimensional from the uniqueness of the Gelfand--Graev model of $\varrho$.
		Here, $\mcu_M=\mcu\cap\mcm$ is a unipotent radical of a Borel subgroup of $\mcm$. Combining with \eqref{eqlambdaMwhittaker}, we have shown that
		\begin{equation}
			\begin{aligned}
				\mrhom_{J_M^g\cap U_M}((\lambda_M^0)^g,\psi_M)\cong\begin{cases}
					\mbc&\quad\text{if}\ g\in\bs{J}_M,\\
					0&\quad\text{if}\ g\notin\bs{J}_M.
				\end{cases}
			\end{aligned}
		\end{equation}
		As a by-product, we have $\mrhom_{J_M\cap U_M}(\kappa_M,\chi)\cong \mbc.$
		
		Now we prove (2). Let $h=g^{-1}$. Then we have
		\begin{equation*}
			\begin{aligned}
				\mrhom_{ \ol{J_M^g}\cap U_M}(\lambda_M^g,\psi_M)&\cong \mrhom_{ \ol{J_M}}(\lambda_M,(\mrInd_{\ol{J_M^g}\cap U_M}^{\ol{J_M^g}}\psi_M)^h)\\&\cong\mrhom_{ \ol{J_M}}(\lambda_M,\mrInd_{(\ol{J_M^g}\cap U_M)^h}^{\ol{J_M}}((\psi_M\rest_{\ol{J_M^g}\cap U_M})^h))	
			\end{aligned}	
		\end{equation*}
		Since $J_M\cap U_M^h$ is a pro-$p$-group that admits a unique splitting,  $(\psi_M\rest_{\ol{J_M^g}\cap U_M})^h$ equals the representation $\psi_M^h$ of $(\ol{J_M^g}\cap U_M)^h=\ol{J_M}\cap U_M^h$. Thus we have
		\begin{equation*}
			\begin{aligned}\mrhom_{ \ol{J_M^g}\cap U_M}(\lambda_M^g,\psi_M)&\cong\mrhom_{J_M}(\lambda_M^0,\mrInd_{J_M\cap U_M^h}^{J_M}(\psi_M^h\rest_{J_M\cap U_M^h}))\\
				&\cong\mrhom_{J_M^g}((\lambda_M^0)^g,\mrInd_{J_M^g\cap U_M}^{J_M^g}(\psi_M\rest_{J_M^g\cap U_M}))\\
				&\cong\mrhom_{J_M^g\cap U_M}((\lambda_M^0)^g,\psi_M)
			\end{aligned}	
		\end{equation*}
		Using \eqref{eqlambdaMwhittaker}, we get the statement (2).
		
		Using the Frobenius reciprocity and Mackey theorem and statement (2), we have
		\begin{equation}\label{eqsimpletypewhittaker2}
			\begin{aligned}
				\mrhom_{\ol{M}}(\mrind_{\ol{J_M}}^{\ol{M}}(\lambda_M),\mrind_{U_M}^{\ol{M}}(\psi_M))&\cong \bigoplus_{g\in J_M\backslash M/U_M}\mrhom_{\ol{J_M^g}\cap U_M}(\lambda_M^g,\psi_M)\\
				& \cong \bigoplus_{g\in J_M\backslash \bs{J}_M}\mrhom_{\ol{J_M^g}\cap U_M}(\lambda_M^g,\psi_M)\\
				&\cong\mrhom_{\ol{\bs{J}_M}}(\mrind_{J_M}^{\ol{\bs{J}_M}}(\lambda_M),\mrind_{ \bs{J}_M\cap U_M}^{\ol{\bs{J}_M}}(\psi_M)),
			\end{aligned}
		\end{equation}
		which implies statement (3). Finally, statement (4) is also direct from statement (2). 
		
	\end{proof}

	
	\begin{lemma}\label{lemmaHMVMsupportadd}
		
		Given $g\in\bs{J}_M$ and $g_0\in N_{M}(\lambda_M)$. Let $0\neq \varphi_g\in \mcv_M^{\lambda_M}$ such that $\mrsupp(\varphi_g\rest_{\lambda_M})=U_Mg\ol{J_M}$, and $0\neq\phi_{g_0}\in\mch(\ol{M},\lambda_M)$ that is supported on $\ol{J_M}g_0\ol{J_M}=\ol{J_M}g_0$. Then the convolution product $\phi_{g_0}\ast\varphi_{g}$ is a non-zero element in $\mcv_M^{\lambda_M}$ such that $\mrsupp(\phi_{g_0}\ast\varphi_{g})= U_Mgg_0^{-1}\ol{J_M}$.
		
	\end{lemma}
	
	\begin{proof}
		
		First of all, since $\phi_{g_0}$ is invertible, $\phi_{g_0}\ast\varphi_{g}$ is non-zero. Choose $0\neq v\in \lambda_M$. By definition, $$\phi_{g_0}\ast\varphi_{g}(v)=\int_{m\in\ol{M}}R(m)\varphi_{g}(\phi_{g_0}^{\vee}(m)v)dm,$$
		where $R(\cdot)$ denotes the right $\ol{M}$-action on $\mrind_{U_M}^{\ol{M}}(\psi_M)$. 
		Since $\phi_{g_0}^{\vee}$ is supported on $g_0^{-1}\ol{J_M}$, the $m$ in the integral indeed ranges over $g_0^{-1}\ol{J_M}$. For such $m$ in the integral, $\varphi_{g}(\phi_{g_0}^{\vee}(m)v)$ is a function in $\mrind_{U_M}^{\ol{M}}(\psi_M)$ that is supported on $U_Mg\ol{J_M}$, and $R(m)\varphi_{g}(\phi_{g_0}^{\vee}(m)v)$ is supported on $U_Mgg_0^{-1}\ol{J_M}$. Here, we used the fact that $\bs{J}_M$ normalizes $J_M$. Thus $\mrsupp(\phi_{g_0}\ast\varphi_{g}(v))= U_Mgg_0^{-1}\ol{J_M}$. 
		
	\end{proof}

	\begin{proposition}\label{propVMlambdaMfree}
		
		$\mcv_M^{\lambda_M}$ is a free $\mch(\ol{M},\lambda_M)$-module of rank $\car{\bs{J}_M/N_M(\lambda_M)}$.
		
	\end{proposition}
	
	\begin{proof}
		
		Let $g_1,\dots,g_k$ be a sequence of representatives in $\bs{J}_M/N_M(\lambda_M)$, where $k=\car{\bs{J}_M/N_M(\lambda_M)}$. Using Lemma \ref{lemmaVlambdasupport}.(4), for each $i=1,\dots,k$ we define a non-zero function $\varphi_{i}\in \mcv_M^{\lambda_M}$ satisfying $\mrsupp(\varphi_{i}\rest_{\lambda_M})\subset U_Mg_i\ol{J_M}$. Using Proposition \ref{propintertwinigset} and Lemma \ref{lemmaHMVMsupportadd}, 
		$\mch(\ol{M},\lambda_M)\ast\varphi_i$ is a free $\mch(\ol{M},\lambda_M)$-module, consisting exactly of those $\varphi\in\mcv_M^{\lambda_M}$ such that $\mrsupp(\varphi\rest_{\lambda_M})\subset U_Mg_iN_{\ol{M}}(\lambda_M)$. Using Lemma \ref{lemmaVlambdasupport}.(3), we have a related isomorphism 
		$$\mcv_M^{\lambda_M}\cong \bigoplus_{i=1}^{k}\mch(\ol{M},\lambda_M)\ast\varphi_i$$
		of $\mch(\ol{M},\lambda_M)$-modules. So the proof is finished.
		
	\end{proof}
	
	\begin{corollary}\label{corcuspidalWhittaker}
		
		Let $\pi$ be an irreducible representation of $\ol{M}$ that contains $\lambda_M$. Then the Whittaker dimension of $\pi$ is $\car{\mcx(\lambda)}=\car{\bs{J}_M/N_M(\lambda_M)}$.
		
	\end{corollary}
	
	\begin{proof}
		
		We need to calculate the dimension of the vector space
		$$\mrhom_{\ol{M}}(\mcv_M,\pi)\cong\mrhom_{\mch(\ol{M},\lambda_M)}(\mcv_M^{\lambda_M},\pi^{\lambda_M}).$$
		Since $\mch(\ol{M},\lambda_M)$ is abelian and the irreducible representation $\pi$ contains $\lambda_M$, we have that $\pi^{\lambda_M}$ is a character of $\mch(\ol{M},\lambda_M)$ of multiplicity one. Using Proposition \ref{propVMlambdaMfree}, we have $$\mrhom_{\mch(\ol{M},\lambda_M)}(\mcv_M^{\lambda_M},\pi^{\lambda_M})\cong\car{\bs{J}_M/N_M(\lambda_M)}\cdot\mrhom_{\mch(\ol{M},\lambda_M)}(\mch(\ol{M},\lambda_M),\pi^{\lambda_M})\cong\mbc^{\car{\bs{J}_M/N_M(\lambda_M)}}.$$
	\end{proof}
	
	\section{$\mch_0$-module structure of $\mcv^{\lambda'}$}\label{sectionH0module}
	
	In this section, we consider $\mcv^{\lambda'}$ as an $\mch(\ol{G},\lambda')$-module. In particular, since $\mch_0\cong\mch(\ol{ J_{\text{max}}},\lambda')$ is the related finite Hecke algebra as a subalgebra of $\mch(\ol{G},\lambda')$, we would like to study the $\mch_0$-module structure of $\mcv^{\lambda'}$.
	
	Our goal is to study the space
	$$\mcv^{\lambda'}=\mrhom_{\ol{G}}(\mrind_{\ol{J'}}^{\ol{G}}(\lambda'),\mrind_{U}^{\ol{G}}(\psi)).$$
	Using the Frobenius reciprocity and Mackey theorem, we have
	\begin{equation*}
		\begin{aligned}
			\mrhom_{\ol{G}}(\mrind_{\ol{J'}}^{\ol{G}}(\lambda'),\mrind_{U}^{\ol{G}}(\psi))\cong\mrhom_{\ol{J'}}(\lambda',\mrind_{U}^{\ol{G}}(\psi)\rest_{\ol{J'}})&\cong\bigoplus_{g\in U\backslash G/J'}\mrhom_{\ol{J'}}(\lambda',\mrind_{U^g\cap \ol{J'}}^{\ol{J'}}(\psi^g))\\
			&\cong\bigoplus_{g\in U\backslash G/J'}\mrhom_{\ol{J'}\cap U^g}(\lambda',\psi^g).
		\end{aligned}
	\end{equation*}
	We study those $g$ such that 
	$$\mrhom_{\ol{J'}\cap U^g}(\lambda',\psi^g)\cong\mrhom_{\ol{J'}\cap U^g}(\kappamax,\chi^g)\otimes\mrhom_{\ol{J'}\cap U^g}(\rho',\psi_b^g)$$
	is non-zero.
	
	\begin{lemma}\label{lemmathetachidist}
		
		For $g\in U B^\times  J_{\text{max}}$, we have $\chi^g\rest_{ H_{\text{max}}^{1}\cap U^g}=\theta_{\text{max}}\rest_{ H_{\text{max}}^{1}\cap U^g}.$
		
	\end{lemma}
	
	\begin{proof}
		
		Using the Iwasawa decomposition $B^\times=U_B T_BU(\mfb_{\text{max}})$ and the fact that $U(\mfb_{\text{max}})\subset J_{\text{max}}$ normalizes $\theta_{\text{max}}$, we may without loss of generality assume that $g\in T_B$. Recall that we have
		$$ H_{\text{max}}^{1}=( H_{\text{max}}^{1}\cap N_0^{-})( H_{\text{max}}^{1}\cap M_0)( H_{\text{max}}^{1}\cap N_0)$$
		and
		$$\theta_{\text{max}}\rest_{ H_{\text{max}}^{1}\cap N_0^{-}}=\theta_{\text{max}}\rest_{ H_{\text{max}}^{1}\cap N_0}=1.$$
		Since $g$ normalizes $N_0^-$ and $\chi$ is trivial on $N_0^-$, we get
		$$\chi^g\rest_{ H_{\text{max}}^{1}\cap N_0^{-}\cap U^g}=\chi^g\rest_{ H_{\text{max}}^{1}\cap (N_0^{-}\cap U)^g}=1=\theta_{\text{max}}\rest_{ H_{\text{max}}^{1}\cap (N_0^{-}\cap U)^g}.$$
		Since $g$ normalizes $N_0$, and $N_0\cap U=\{1\}$, 
		we get
		$$\chi^g\rest_{ H_{\text{max}}^{1}\cap N_0\cap U^g}=\chi^g\rest_{ H_{\text{max}}^{1}\cap (N_0\cap U)^g}=1=\theta_{\text{max}}\rest_{ H_{\text{max}}^{1}\cap (N_0\cap U)^g}.$$
		Finally, since $g$ normalizes $ H_{\text{max}}^{1}\cap M_0$ and $\theta_{\text{max}}\rest_{ H_{\text{max}}^{1}\cap M_0}$, and $\chi\rest_{  H_{\text{max}}^{1}\cap U}=\theta_{\text{max}}\rest_{  H_{\text{max}}^{1}\cap U}$, we have
		$$\chi^g\rest_{ H_{\text{max}}^{1}\cap M_0\cap U^g}=\chi^g\rest_{( H_{\text{max}}^{1}\cap M_0\cap U)^g}=\theta_{\text{max}}^g\rest_{( H_{\text{max}}^{1}\cap M_0\cap U)^g}=\theta_{\text{max}}\rest_{ H_{\text{max}}^{1}\cap M_0\cap U^g}.$$
		So we finish the proof.
		
	\end{proof}

	\begin{corollary}
		
		For $g\in U B^\times J_{\text{max}}$, we have $$\mrhom_{J_{\text{max}}^{1}\cap U^g}(\eta_{\text{max}},\chi^g)=\mrhom_{J'\cap U^g}(\kappamax,\chi^g)=\mrhom_{ J_{\text{max}}\cap U^g}(\kappamax,\chi^g)\cong \mbc.$$
		
	\end{corollary}
	
	\begin{proof}
		
		Without loss of generality, we may assume that $g\in B^\times$. Using the Iwasawa decomposition $B^\times=U(\mfb_{\text{max}})T_B U_B$, we may further assume that $g\in T_B$.
		
		First we show that \begin{equation}\label{eqetamaxdist}
			\mrhom_{J_{\text{max}}^{1}\cap U^g}(\eta_{\text{max}},\chi^g)\cong\mbc.
		\end{equation}
		
		\begin{lemma}\label{lemmaJhcapUdecom}
			
			We have $H_{\text{max}}^{1}\cap U^g=(H_{\text{max}}^{1}\cap N_0^-)(H_{\text{max}}^{1}\cap M_0\cap U)^g$, $J_{\text{max}}^{1}\cap U^g=(J_{\text{max}}^{1}\cap N_0^-)(J_{\text{max}}^{1}\cap M_0\cap U)^g$ and $ J_{\text{max}}\cap U^g=( J_{\text{max}}\cap N_0^-)( J_{\text{max}}\cap M_0\cap U)^g$.
			
		\end{lemma}
		
		\begin{proof}
			
			We take the last equation as an example. By the Iwahori decomposition and the fact that $g$ normalizes $M_0$, $N_0^-$ and $ J_{\text{max}}\cap M_0$, we have
			$$ J_{\text{max}}\cap U^g=( J_{\text{max}}\cap N_0^{-g})( J_{\text{max}}\cap M_0^g\cap U^g)=( J_{\text{max}}\cap N_0^{-})( J_{\text{max}}\cap M_0\cap U)^g.$$
			
		\end{proof}
		
		We claim that 
		\begin{equation}\label{eqHUJUHJ}
			[J_{\text{max}}^{1}\cap U^g:H_{\text{max}}^{1}\cap U^g]=[J_{\text{max}}^{1}:H_{\text{max}}^{1}]^{1/2}.
		\end{equation}
		Indeed, using the Iwahori decomposition and the above lemma, we have the decomposition (\emph{cf.} \cite{paskunas2008realization}*{Proof of Theorem 2.6})
		$$J_{\text{max}}^{1}/H_{\text{max}}^{1}\cong(J_{\text{max}}^{1}\cap N_0^-/(H_{\text{max}}^{1}\cap N_0^-)\times (J_{\text{max}}^{1}\cap M_0)/(H_{\text{max}}^{1}\cap M_0)\times(J_{\text{max}}^{1}\cap N_0^-)/(H_{\text{max}}^{1}\cap N_0^-)$$
		of symplectic spaces over $\mbf_p$, where the three subspaces on the right-hand side are orthogonal to each other, and the first and the third subspaces are totally isotropic and of the same dimension. Moreover, $J_{\text{max}}^{1}\cap M_0\cap U/H_{\text{max}}^{1}\cap M_0\cap U$ is a maximal totally isotropic subspace in $J_{\text{max}}^{1}\cap M_0/H_{\text{max}}^{1}\cap M_0$. As a result \eqref{eqHUJUHJ} is proved.
		
		Using Lemma \ref{lemmathetachidist}, \eqref{eqetathetamax},  \eqref{eqHUJUHJ}, Frobenius reciprocity and Mackey theory and the fact that $ J_{\text{max}}^{1}$ normalizes $ H_{\text{max}}^{1}$ and $\theta_{\text{max}}$, we have
		\begin{equation*}
			\begin{aligned}
				[ J_{\text{max}}^{1}\cap U^g: H_{\text{max}}^{1}\cap U^g]\cdot\mrhom_{J_{\text{max}}^{1}\cap U^g}(\eta_{\text{max}},\chi^g)&\cong\bigoplus_{g'\in H_{\text{max}}^{1}\backslash J_{\text{max}}^{1}/J_{\text{max}}^{1}\cap U^g}\mrhom_{H_{\text{max}}^{1g'}\cap U^g}(\theta_{\text{max}}^{g'},\chi^g)\\
				&\cong[ J_{\text{max}}^{1}:( J_{\text{max}}^{1}\cap U^g) H_{\text{max}}^{1}]\cdot \mrhom_{H_{\text{max}}^{1}\cap U^g}(\theta_{\text{max}},\chi^g)\\
				&\cong\mbc^{[ J_{\text{max}}^{1}\cap U^g: H_{\text{max}}^{1}\cap U^g]}.
			\end{aligned}	
		\end{equation*}
		So we proved \eqref{eqetamaxdist}.
		
		Then we show that for any $g\in T_B$,
		\begin{equation}\label{eqkappamaxneq0}
			\mrhom_{ J_{\text{max}}\cap U^g}(\kappamax,\chi^g)\neq 0.
		\end{equation}
		Using Lemma \ref{lemmaJhcapUdecom},  the fact that $g$ normalizes $N_0^-$, $M_0$ and $\kappamax\rest_{ J_{\text{max}}\cap M_0}$, and $\chi^g\rest_{N_0^-}=\chi\rest_{N_0^-}=1$, we need to show that
		$$\mrhom_{ J_{\text{max}}\cap N_0^-}(\kappamax,\chi^g)=\mrhom_{ J_{\text{max}}\cap N_0^-}(\kappamax,\chi)\neq 0$$
		and
		$$\mrhom_{ J_{\text{max}}\cap M_0^g\cap U^g}(\kappamax,\chi^g)= \mrhom_{( J_{\text{max}}\cap M_0\cap U)^g}(\kappamax^g,\chi^g)\cong \mrhom_{ J_{\text{max}}\cap M_0\cap U}(\kappamax,\chi)\neq 0.$$
		Thus
		we need to show that
		$$\mrhom_{ J_{\text{max}}\cap U}(\kappamax,\chi)\neq 0,$$
		which follows from
		Lemma \ref{lemmaVlambdasupport}.(1) by considering the special case $M=G$. 

		Finally the corollary follows from \eqref{eqetamaxdist}, \eqref{eqkappamaxneq0} and the obvious inclusions
		$$\mrhom_{J_{\text{max}}^{1}\cap U^g}(\eta_{\text{max}},\chi^g)\supset\mrhom_{J'\cap U^g}(\kappamax,\chi^g)\supset\mrhom_{ J_{\text{max}}\cap U^g}(\kappamax,\chi^g).$$
		
	\end{proof}
	
	We digress for a moment to discuss the converse of Lemma \ref{lemmathetachidist}, which will be used in Section \ref{sectionHmodulemcv} (More precisely, Proposition \ref{propPhitiPhi'tivarphiti}).

	\begin{proposition}\label{propthetachieqconverse}
		
		For $g\in G$ such that $\chi^g\rest_{ H_{\text{max}}^{1}\cap U^g}=\theta_{\text{max}}\rest_{ H_{\text{max}}^{1}\cap U^g}$, we have $g\in U B^\times  J_{\text{max}}$.
		
	\end{proposition}
	
	\begin{proof}
		
		Using the containment
		$$ H_{\text{max}}^{1}\cap N_0^{-g}\subset H_{\text{max}}^{1}\cap U^g$$
		and the facts that $\chi\rest_{N_0^-}=1$,
		the condition $\chi^g\rest_{ H_{\text{max}}^{1}\cap U^g}=\theta_{\text{max}}\rest_{ H_{\text{max}}^{1}\cap U^g}$ implies that
		$$\theta_{\text{max}}\rest_{ H_{\text{max}}^{1}\cap N_0^{-g}}=1.$$
		
		We need the following lemma.
		
		\begin{lemma}[\cite{secherre2016block}*{Statement (5.6) in Proposition 5.6}]\label{lemmatechnical}
			
			For $g\in  G$, we have  $\theta_{\text{max}}\rest_{ H_{\text{max}}^{1}\cap N_0^{-g}}=1$ if and only if $g\in N_0^{-}M_0J_{\text{max}} $.	
			
		\end{lemma}

		Using the above lemma, we have $g\in N_0^{-}M_0J_{\text{max}}$.  Replacing $g$ by another representative in the double coset $UgJ_{\text{max}}$, we may assume without loss of generality that $g\in M_0$.
		
		\begin{lemma}
			
			For $m\in M_0$, we have $\chi^{m}\rest_{ H_{\text{max}}^{1}\cap U_{M_0}^{m}}=\theta_{\text{max}}\rest_{ H_{\text{max}}^{1}\cap U_{M_0}^{m}}$ if and only if $m\in U_{M_0}T_B( J_{\text{max}}\cap M_0)$.
			
		\end{lemma}
		
		\begin{proof}
			
			Let $\theta_{M_0}=\theta_{\text{max}}\rest_{ H_{\text{max}}^{1}\cap M_0}$, $\eta_{M_0}=\eta_{\text{max}}\rest_{ J_{\text{max}}^{1}\cap M_0}$ and $\kappa_{M_0}=\kappa_{\text{max}}\rest_{ J_{\text{max}}\cap M_0}$.
			The intertwining set  and normalizer of $\kappa_{M_0}$ in $M_0$ is $$\bs{J}_{M_0}:=T_B( J_{\text{max}}\cap M_0).$$ Let $\bs{\lambda}_{M_0}$ be an irreducible of $\bs{J}_{M_0}$ extending $\kappa_{M_0}$. Then the normalizer of $\bs{\lambda}_{M_0}$ is $\bs{J}_{M_0}$, implying that the compact induction $\pi_{M_0}=\mrind_{\bs{J}_{M_0}}^{M_0}\bs{\lambda}_{M_0}$ is an irreducible cuspidal representation of $M_0$. Using the multiplicity one result of the Whittaker model of $\pi_{M_0}$, the fact that $\chi\rest_{U_{M_0}}=\psi\rest_{U_{M_0}}$ is a generic character of $M_0$, and Frobenius reciprocity and the Mackey formula, we have
			$$\mbc\cong\mrhom_{U_{M_0}}(\pi_{M_0},\chi)\cong\bigoplus_{m\in \bs{J}_{M_0}\backslash M_0/U_{M_0}}\mrhom_{\bs{J}_{M_0}^m\cap U_{M_0}}(\bs{\lambda}_{M_0}^m,\chi).$$
			Using the proof of Lemma \ref{lemmaVlambdasupport}.(1) with $M$ replaced by $M_0$ and the fact that $\bs{J}_{M_0}\cap U_{M_0}=J_{M_0}\cap U_{M_0}$, we may show that $$\mrhom_{\bs{J}_{M_0}\cap U_{M_0}}(\bs{\lambda}_{M_0},\chi)=\mrhom_{J_{M_0}\cap U_{M_0}}(\kappa_{M_0},\chi)\neq 0,$$
			and thus
			$$\mrhom_{\bs{J}_{M_0}^m\cap U_{M_0}}(\bs{\lambda}_{M_0}^m,\chi)=0$$
			for $m\notin \bs{J}_{M_0}U_{M_0}$. Restricting to $( H_{\text{max}}^{1}\cap M_0)^m\cap U_{M_0}$ and noticing that $\bs{\lambda}_{M_0}\rest_{ H_{\text{max}}^{1}\cap M_0}$ is a multiple of $\theta_{M_0}$, we have that $\chi^{m}\rest_{ H_{\text{max}}^{1}\cap U_{M_0}^{m}}=\theta_{\text{max}}\rest_{ H_{\text{max}}^{1}\cap U_{M_0}^{m}}$ if and only if $m\in U_{M_0}\bs{J}_{M_0}$.
			
		\end{proof}
		
		Using the above lemma for $m=g$, we have $g\in U_{M_0}T_B( J_{\text{max}}\cap M_0)\subset UB^\times J_{\text{max}}$, which finishes the proof of Proposition \ref{propthetachieqconverse}.
		
	\end{proof}
	
	Come back to our original discussion. Recall that we have the Iwasawa decomposition
	$$B^\times=U_BT_B U(\bmax).$$
	So we may essentially assume $g=tx^{-1}$ for $x\in U(\bmax)$ and $t\in T_B$. We study when the space
	\begin{equation}\label{eqrhopsib}
		\mrhom_{\ol{J'}\cap U^g}(\rho',\psi_b^g)
	\end{equation}
	is non-zero. Restricting to $U^1(\bmax)\cap U_B$, we necessarily have
	$$\mrhom_{U^1(\bmax)\cap U_B^t}(1,\psi_b^t)=\mrhom_{U^1(\bmax)\cap U_B}(1,\psi_b^t)\neq 0.$$
	
	\begin{lemma}\label{lemmapsibnonzterocond}
		
		For $t\in T_B$, the restriction $\psi_b^t\rest_{U^1(\bmax)\cap U_B}$ is trivial if and only if $t\in T_{B,\leq}$, where 
		$$T_{B,\leq}:=\{\mrdiag(b_1,\dots,b_m)\in T_B\mid v_E(b_1)\leq v_E(b_2)\leq\dots\leq v_E(b_m)\}.$$
		Moreover, if the above condition is satisfied, the restriction $\psi_b^t\rest_{J_{\text{max}}^{1}\cap U^t}$ is also trivial.
		
	\end{lemma}
	
	\begin{proof}
		
		Using Theorem \ref{thmsimpletypewhittaker}.(5), we have $\psi_b(u')=\psi(u')=\psi_E(\sum_{i=1}^{m-1}u_{i+1,i}')$ for any $u'=(u_{ij}')\in U_B$, where $U_B$ is identified with the set of the lower triangular unipotent matrices in $B^\times\cong\mrgl_m(E)$. Write $t=\mrdiag(b_1,\dots,b_m)\in T_B$. Given $u'=(u_{ij}')\in U^1(\bmax)\cap U_B$, we have $\psi_b^t(u')=\psi_E(\sum_{i=1}^{m-1}u_{i+1,i}'b_{i+1}/b_{i})=0$. Since $u_{i+1,i}'\in\mfp_E$ are arbitrary and $\psi_E$ is of conductor $\mfp_E$, we have $t\in T_{B,\leq}$. Conversely, $t\in T_{B,\leq}$ also implies that $\psi_b^t\rest_{U^1(\bmax)^t\cap U_B}$ is trivial.
		
		For the last claim, we have the Iwahori decomposition
		$$(J_{\text{max}}^{1}\cap U^t)=(J_{\text{max}}^{1}\cap N_0^{-t})(J_{\text{max}}^{1}\cap M_0\cap U^t)=(J_{\text{max}}^{1}\cap N_0^-)(J_{\text{max}}^{1}\cap M_0\cap U)^t.$$
		We have $J_{\text{max}}^{1}\cap M_0\cap U\subset U^1(\amax)\cap M_0\cap U$. 
		Also, we have $J_{\text{max}}^{1}\cap N_0^-\subset U^1(\amax)\cap N_0^-\subset (U^1(\amax)\cap N_0^-)^t$. Using the fact that $\psi_b$ is trivial on $U^1(\amax)\cap U$, we finish the proof.
		
	\end{proof}
	
	As a result, \eqref{eqrhopsib} is non-zero implies that $t\in T_{B,\leq}$. In this case, we have
	$$\mrhom_{\ol{J'}\cap U}(\rho',\psi_b^g)=\mrhom_{\ol{U(\mfb)}\cap U_B}(\rho',\psi_b^g)\cong\mrhom_{\ol{U(\mfb)^x}\cap U_B}(\rho'^x,\psi_b^t)=\mrhom_{\mcp^x\cap \mcu}(\varrho^x,\ul{\psi}_b^{t}).$$

	\begin{lemma}\label{lemmacondtx-1nonzero}
		
		The space $\mrhom_{\mcp^x\cap \mcu}(\varrho^x,\ul{\psi}_b^t)$ is non-zero if and only if the following two conditions are satisfied: 
		\begin{itemize}
			\item 
			Write $t=\mrdiag(b_1,\dots,b_m)\in T_{B,\leq}$. Let $(m_1,\dots,m_l)$ be the composition of $m$ such that
			$$v_E(b_{i})<v_E(b_{i+1})\ \text{if and only if}\ i=m_1+\dots+m_j\ \text{for some}\ j=1,\dots,l-1.$$
			Then for each $j=1,\dots,l$, we have $m_j=m_0k_j$ for some $k_j\in\mbn$. Equivalently, for $\mcp_t$ the upper block triangular parabolic subgroup of $\mcg$ whose Levi subgroup is $$\mcm_t:=\mrgl_{m_1}(\bs{l})\times\dots\times\mrgl_{m_l}(\bs{l}),$$
			we have $\mcp\subset\mcp_t$ and $\mcm\subset\mcm_t$;
			\item $x$ is a representative of a $\mcp$-$\mcp_t$ double coset, such that among all the elements $\mcp x\mcp_t$, it is of the smallest length with respect to $\mcp\backslash \mcg/\mcp=\mcp\backslash \mcg/\mcn\cong \mfS_k$. 
			
		\end{itemize}
		In this case, it is of dimension one.
		
	\end{lemma}
	
	\begin{proof}
		
		Replacing $x$ by another representative in the same double coset $\mcp$-$\mcn^-$, we may assume without generality that $x$ is in $$W_0(\mfb)=\{(\delta_{\sigma(i)j}I_{m_0})_{ 1\leq i,j\leq k}\mid \sigma\in\mfS_k \}\subset\mcg,$$ which in particular normalizes $\mcm$. Restricting to $\mcn^x\cap\mcu$ and $\mcm^x\cap\mcu=\mcm\cap\mcu$, the condition $\mrhom_{\mcp^x\cap \mcu}(\varrho^x,\ul{\psi}_b^t)\neq 0$ is equivalent to  $$\ul{\psi}_b^t\rest_{\mcn^x\cap\mcu}=1\quad\text{and}\quad\ul{\psi}_b^t\rest_{\mcm\cap \mcu}\ \text{is a generic character of}\ \mcm\cap \mcu.$$
		
		From the expression of $\ul{\psi}_b^t$, the restriction  $\ul{\psi}_b^t\rest_{\mcm\cap \mcu}$ is a generic character of $\mcm\cap \mcu$ if and only if $\{m_1,\dots,m_l\}$ is a subset of $\{m_0,2m_0,\dots,(k-1)m_0\}$, or equivalently, for each $i$, we have $m_i=m_0k_i$ for some $k_i\in\mbn$. Now we assume this condition.
		
		Write $x=(\delta_{\sigma(i)j}I_{m_0})_{ 1\leq i,j\leq k}$ for some $\sigma\in\mfS_k$. Then, it is easy to see that $\ul{\psi}_b^t\rest_{\mcn^x\cap\mcu}=1$ is equivalent to the fact that $\sigma(j+1)>\sigma(j)$ for $j\neq k_1,\dots,k_l$. Otherwise, the related unipotent group $\mcu_{\alpha_{jm_0+1,jm_0}}$ related to the simple root $\alpha_{jm_0+1,jm_0}$ is in $\mcn^x\cap\mcu$ and $\ul{\psi}_b^t\rest_{\mcu_{\alpha_{jm_0+1,jm_0}}}\neq 1$. However, the above condition for $\sigma$ is equivalent to saying that $x$ is a representative of a $\mcp$-$\mcp_t$ double coset whose length in $\mcp\backslash \mcg/\mcp=\mcp\backslash \mcg/\mcn\cong \mfS_k$ is minimal. 
		
		Finally, the multiplicity one result follows from the uniqueness of Gelfand--Graev model over finite fields (\emph{cf.} \cite{silberger2000characters}*{Corollary 5.4}).
		
	\end{proof}
	
	\begin{corollary}
		
		Given $t$ as in the first condition of the lemma. Let $\ul{R}_t$ be a set of representatives of the $\mcp$-$\mcp_t$ double cosets, such that each representative has the minimal length with respect to $\mcp\backslash \mcg/\mcp=\mcp\backslash \mcg/\mcn\cong \mfS_k$ among all the elments in the same double coset. Then we have
		$$\ul{\mcv}^\varrho:=\mrhom_{\mcg}(\mrind_\mcp^\mcg(\varrho),\mrind_\mcu^\mcg(\ul{\psi}_b^t))\cong\bigoplus_{x\in \ul{R}_t}\mrhom_{\mcp^x\cap \mcu}(\varrho^x,\ul{\psi}_b^t)\cong\mbc^{\car{\ul{R}_t}},$$
		where $\ul{R}_t$ is of cardinality $$\car{\mcp\backslash \mcg/\mcp_t}=\car{W(\mcg,\mcm)/W(\mcm_t,\mcm)}=\car{\mfS_k/(\mfS_{k_1}\times\dots\times \mfS_{k_1})}=k!/(k_1!\cdot\dots\cdot k_{l}!).$$ 	
		Moreover, as a $\mch(\mcg,\varrho)$-module $\ul{\mcv}^\varrho$ is isomorphic to $\mch(\mcg,\varrho)\otimes_{\mch(\mcm_t,\varrho)}\varepsilon_{\mch(\mcm_t,\varrho)}$. Here, $\mch(\mcg,\varrho)$ (resp. $\mch(\mcm_t,\varrho)$) is the finite Hecke algebra whose Weyl group is $W(\mcg,\mcm)$ (resp. $W(\mcm_t,\mcm)$), and $\varepsilon_{\mch(\mcm_t,\varrho)}$ is the sign character of $\mch(\mcm_t,\varrho)$.
		
	\end{corollary}
	
	\begin{proof}
		
		We only need to explain the last claim. Like the $p$-adic case, in the finite case we have the following commutative diagram (\emph{cf.} \eqref{eqMrho} for the definition of $\bs{\mathrm{M}}_{\varrho}$)
		$$\xymatrix{
			\mrrep(\mcg)  \ar[r]^-{\bs{\mathrm{M}}_{\varrho}} & \mrmod(\mch(\mcg,\varrho)) \\ 
			\mrrep(\mcm_t) \ar[r]_-{\bs{\mathrm{M}}_{\varrho}} \ar[u]^-{\mrind_{\mcp_t}^{\mcg}} & \mrmod(\mch(\mcm_t,\varrho))  \ar[u]_-{\mrind_{\mch(\mcm_t,\varrho)}^{\mch(\mcg,\varrho)}}
		}$$ 
		and the isomorphism 
		$$\mrind_{\mcp_t}^{\mcg}(\mrind_{\mcu\cap \mcm_t}^{\mcm_t}(\ul{\psi}_b^t))\cong \mrind_\mcu^\mcg(\ul{\psi}_b^t).$$
		Thus, we only need to show that $$\ul{\mcv}_{\mcm_t}^\varrho:=\mrhom_{\mcm_t}(\mrind_{\mcp\cap\mcm_t}^{\mcm_t}(\varrho),\mrind_{\mcu\cap \mcm_t}^{\mcm_t}(\ul{\psi}_b^t))$$ 
		being realized as a $\mch(\mcm_t,\varrho)$-module is isomorphic to the sign character $\varepsilon_{\mch(\mcm_t,\varrho)}$. 
		Indeed, by \cite{silberger2000characters}*{Corollary 5.7} the generalize Steinberg representation, denoted by $\mathrm{St}(\mcm_t,\varrho)$, is the unique irreducible generic subrepresentation of $\mrind_{\mcp\cap\mcm_t}^{\mcm_t}(\varrho)$. By \cite{silberger2000characters}*{Proposition 4.2}, we have $$\ul{\mcv}_{\mcm_t}^\varrho\cong\mrhom_{\mcm_t}(\mrind_{\mcp\cap\mcm_t}^{\mcm_t}(\varrho),\mathrm{St}(\mcm_t,\varrho))\cong \varepsilon_{\mch(\mcm_t,\varrho)}$$
		as a $\mch(\mcm_t,\varrho)$-module.
	\end{proof}
	
	Following the above lemma and corollary, we define $T(\mfb)_{\leq}=T(\mfb)\cap T_{B,\leq}$. For $t\in T(\mfb)_{\leq}$, let $\ul{R}_t$ be a set of representatives the $\mcp$-$\mcp_t$ double cosets as above. Let $R_t$ to be a set of elements in $U(\bmax)$, whose image in $\mcg\cong U(\bmax)/U^1(\bmax)$ is $\ul{R}_t$. 
	
	We come back to the study of $\mcv^{\lambda'}$. We analyze its subspace
	$$\mcv_B^{\lambda'}:=\bigoplus_{g\in U\backslash UB^\times J'/J'}\mrhom_{\ol{J'}}(\lambda',\mrind_{U^g\cap \ol{J'}}^{\ol{J'}}(\psi^g)).$$
	The following proposition follows directly from our discuss above.
	
	\begin{proposition}
		
		\begin{enumerate}
			\item We have $\mcv_B^{\lambda'}\cong\bigoplus_{t\in T(\mfb)_{\leq}}\bigoplus_{x\in R_t}\mrhom_{\ol{J'^{x}}}(\lambda'^{x},\mrind_{U^t\cap \ol{J'^{x}}}^{\ol{J'^{x}}}(\psi^t)).$
			
			\item For each $t\in T(\mfb)_{\leq}$ and $x\in R_t$ as above, we have  $$\mrhom_{\ol{J'^{x}}}(\lambda'^{x},\mrind_{U^t\cap \ol{J'^{x}}}^{\ol{J'^{x}}}(\psi^t))\cong\mrhom_{ J_{\text{max}}\cap U^t}(\kappamax,\chi^t)\otimes_\mbc \mrhom_{\mcp^x\cap\mcu}(\varrho^x,\ul{\psi}_b^t),$$
			which is isomorphic to $\mbc$. 
			
			\item For $t\in T(\mfb)_{\leq}$, define $$\mcv_{B,t}^{\lambda'}:=\bigoplus_{g\in U\backslash J_{\text{max}}/J'}\mrhom_{\ol{J'}}(\lambda',\mrind_{U^{tg}\cap \ol{J'}}^{\ol{J'}}(\psi^{tg}))\cong\bigoplus_{x\in R_t}\mrhom_{\ol{J'^{x}}}(\lambda'^{x},\mrind_{U^t\cap \ol{J'^{x}}}^{\ol{J'^{x}}}(\psi^t)).$$ 
			Then we have
			$$\mcv_B^{\lambda'}\cong\bigoplus_{t\in T(\mfb)_{\leq}}\mcv_{B,t}^{\lambda'}$$
			and
			$$\mcv_{B,t}^{\lambda'}\cong\mrhom_{ J_{\text{max}}\cap U^t}(\kappamax,\chi^t)\otimes_\mbc \mrhom_{\mcg}(\mrind_\mcp^\mcg(\varrho),\mrind_\mcu^\mcg(\ul{\psi}_b^t)).$$
			Moreover, $\mcv_{B,t}^{\lambda'}$ is a $\mch_0\cong\mch(\ol{ J_{\text{max}}},\lambda')$-module isomorphic to $\mch_0\otimes_{\mch_{\mco_t}}\varepsilon_{\mco_t}$, where $\mco_t$ denotes the $\mfS_k$-orbit of $t$.
		\end{enumerate}
		
	\end{proposition}
	
	The following corollary is also direct. Recall that for $\Phi\in\mcv^{\lambda'}$, the support $\mrsupp(\Phi\rest_{\lambda'})$ is defined as in \eqref{eqsuppphi}.
	
	\begin{corollary}\label{corPhi't}
		
		\begin{enumerate}
			\item For $t\in T(\mfb)_{\leq}$, there exists a non-zero $\Phi_t'\in\mcv^{\lambda'}$ such that $\mrsupp(\Phi_t'\rest_{\lambda'})=\ol{J'}tU$. Such $\Phi_t'$ is unique up to a scalar.
			
			\item The $\mch_0$-module generated by $\Phi_t'$ is isomorphic to $\mch_0\otimes_{\mch_{\mco_t}}\varepsilon_{\mco_t}$, where $\mco_t$ denotes the $\mfS_k$-orbit of $t$ and $\mch_{\mco_t}$ the finite Hecke algebra related to the stabilizer of $\mco_t$ in $\mfS_k$. In particular, for $x\in R_t$ representing $w\in W(\mcg,\mcm)=\mfS_k$, we have $\mrsupp(T_w\ast\Phi_t'\rest_{\lambda'})= \ol{J'}tx^{-1}U$, where $T_w\in\mch_0$ is the operator related to $w$.
			
			\item For different $t_1,\dots,t_k\in T(\mfb)_{\leq}$ we construct the related functions  $\Phi_{t_1}',\dots,\Phi_{t_k}'$ as in (1). Then the vector spaces $\mch_0\ast\Phi_{t_1}',\dots,\mch_0\ast\Phi_{t_k}'$ are linearly independent.
		\end{enumerate}

	\end{corollary}
	
	Finally, we have the following  bonus of Proposition \ref{propthetachieqconverse}, although it will not be used to our discussion later on.
	
	\begin{corollary}
		
		We have $\mcv^{\lambda'}\cong\mcv_B^{\lambda'} $.
		
	\end{corollary}
	
	\begin{proof}
		
		Consider a double coset $UgJ'$ such that $\mrhom_{\ol{J'}}(\lambda',\mrind_{U^g\cap \ol{J'}}^{\ol{J'}}(\psi^g))\neq 0$. Restricting the hom-space to $H_{\text{max}}^{1}$, we have
		$g\in UB^{\times}J_{\text{max}}=UB^{\times}J'$. Thus we finish the proof.
		
	\end{proof}
	
	\section{$\mch$-module structure of $\mcv^{\lambda}$}\label{sectionHmodulemcv}
	
	In this section, we consider $\mcv^{\lambda}$ as an $\mch=\mch(\ol{G},\lambda)$-module and finish the proof of Theorem \ref{thmmain}. We first consider general covers, and then focus on Kazhdan--Patterson covers or the Savin cover.
	
	\subsection{General results}
	
	We first recall that since $\mrind_{J_P}^{J}(\lambda_P)=\lambda$, the map $\mcv^{\lambda_P}\mapsto \mcv^{\lambda}$ is equivariant with respect to the $\mch(\ol{G},\lambda_P)\cong\mch(\ol{G},\lambda)$-action. So we only need to study $\mcv^{\lambda_P}$ as an $\mch(\ol{G},\lambda_P)$-module.
	
	Our goal is to study the space
	$$\mcv^{\lambda_P}=\mrhom_{\ol{G}}(\mrind_{\ol{J_P}}^{\ol{G}}(\lambda_P),\mrind_{U}^{\ol{G}}(\psi)).$$
	Using the Frobenius reciprocity and Mackey theorem, we have
	\begin{equation*}
		\begin{aligned}
			\mrhom_{\ol{G}}(\mrind_{\ol{J_P}}^{\ol{G}}(\lambda_P),\mrind_{U}^{\ol{G}}(\psi))\cong\mrhom_{\ol{J_P}}(\lambda_P,\mrind_{U}^{\ol{G}}(\psi)\rest_{\ol{J_p}})&\cong\bigoplus_{g\in U\backslash G/J_P}\mrhom_{\ol{J_P}}(\lambda_P,\mrind_{U^g\cap \ol{J_P}}^{\ol{J_P}}(\psi^g))\\
			&\cong\bigoplus_{g\in U\backslash G/J_P}\mrhom_{\ol{J_P}\cap U^g}(\lambda_P,\psi^g).
		\end{aligned}
	\end{equation*}
	
	\begin{proposition}\label{propPhit}
		
		For $t\in T(\mfb)_{\leq}$, the space $\mrhom_{\ol{J_P}\cap U^t}(\lambda_P,\psi^t)$ is isomorphic to $\mbc$. As a result, there exists a non-zero $\Phi_t\in\mcv^{\lambda_P}$, unique up to a scalar, such that $\mrsupp(\Phi_t\rest_{\lambda_P})=\ol{J_P}tU$.
		
	\end{proposition}
	
	\begin{proof}
		
		Notice that $t$ normalizes $M,N,N^-,J_M$ and $\kappa_M$.
		First, we show that
		$$\mrhom_{J_P\cap U^t}(\kappa_P,\chi^t)\cong\mbc.$$
		Using the Iwahori decomposition, we only need to show that
		$$\mrhom_{J_{M}\cap U_{M}^t}(\kappa_{M},\chi^t)\cong\mrhom_{J_{M}\cap U_{M}}(\kappa_{M},\chi)\cong\mbc\quad\text{and}\quad \mrhom_{J^1\cap N^{-t}}(\kappa_P,\chi^t)=\mrhom_{J^1\cap N^-}(\kappa_P,1)\neq 0,$$
		where the first statement follows from Lemma \ref{lemmaVlambdasupport}.(1), and the second statement follows from the fact that $\kappa_P\rest_{J^1\cap N^-}=1$ and $\chi\rest_{ N^-}=1$. 
		
		Then, we show that
		$$\mrhom_{\ol{J_P}\cap U^t}(\rho,\psi_b^t)\cong\mbc.$$
		From the expression of $t$ and a similar argument of Lemma \ref{lemmapsibnonzterocond}, it is easily seen that $\psi_b^t$ is trivial on $(J_M^1\cap U)^t\subset (U^1(\mfa)\cap M\cap U)^t$ and $J_P^1\cap N^{-t}\subset U^1(\mfa)\cap N^{-t}\subset U^1(\mfa)\cap N^{-}$, and thus it is trivial on $J_P^1\cap U^t=(J_M^1\cap U)^t(J_P^1\cap N^{-t})$. So we only need to show that
		$$\mrhom_{\ol{U(\mfb)}\cap U_B^t}(\rho,\psi_b^t)\cong\mrhom_{\mcp\cap \mcu^t}(\varrho,\ol{\psi}_b^t)\cong\mbc,$$
		which follows from Lemma \ref{lemmacondtx-1nonzero}.
		
		Finally, we have
		$$\mrhom_{\ol{J_P}\cap U^t}(\lambda_P,\psi^t)\cong\mrhom_{J_P\cap U^t}(\kappa_P,\chi^t)\otimes_\mbc\mrhom_{\ol{J_P}\cap U^t}(\rho,\psi_b^t)\cong\mbc.$$
		
	\end{proof}

	Recall that since $\mrind_{\ol{J_P}}^{\ol{U(\mfb)}\ol{U^1(\mfa)}}(\lambda_P)\cong \mrind_{\ol{J'}}^{\ol{U(\mfb)}\ol{U^1(\mfa)}}(\lambda')$, we have an induced isomorphism of vector spaces
	$$f_{\lambda_P,\lambda'}:\mcv^{\lambda_P}=\mrhom_{\ol{G}}(\mrind_{\ol{J_P}}^{\ol{G}}(\lambda_P),\mrind_{U}^{\ol{G}}(\psi))\rightarrow\mrhom_{\ol{G}}(\mrind_{\ol{J'}}^{\ol{G}}(\lambda'),\mrind_{U}^{\ol{G}}(\psi))=\mcv^{\lambda'},$$
	which is also equivariant with respect to the $\mch(\ol{G},\lambda_P)\cong\mch(\ol{G},\lambda')$-action. In particular, given $g\in G$ and $\Phi\in \mcv^{\lambda_P}$ such that $$\mrsupp(\Phi\rest_{\lambda_P})=Ug\ol{J_P},$$ we necessarily have (\emph{cf.} \eqref{eqisolambdalanbda'}) $$\mrsupp(f_{\lambda_P,\lambda'}(\Phi)\rest_{\lambda'})\subset Ug\ol{U(\mfb)}\ol{U^1(\mfa)}.$$

	Similarly, since $r_N(\mcv)\cong\mcv_M$ as representations of $M$ (\emph{cf.} Proposition \ref{propGGmoduleJacquet}), the Jacquet module functor $r_N$ induces an isomorphism of vector spaces
	$$r_N:\mcv^{\lambda_P}=\mrhom_{\ol{G}}(\mrind_{\ol{J_P}}^{\ol{G}}(\lambda_P),\mrind_{U}^{\ol{G}}(\psi))\rightarrow\mrhom_{\ol{M}}(\mrind_{\ol{J_M}}^{\ol{M}}(\lambda_M),\mrind_{U_M}^{\ol{M}}(\psi_M))=\mcv_M^{\lambda_M}$$
	which is also equivariant with respect to the $t_P(\mch(\ol{M},\lambda_M))\cong\mch(\ol{M},\lambda_M)$-action. 
	
	We have the following important proposition.
	
	\begin{proposition}\label{propPhitiPhi'tivarphiti}
		
		Given $t\in T(\mfb)_{\leq}$, we define $\Phi_t'$, $\Phi_t$ and $\varphi_t$ as in Corollary \ref{corPhi't}, Proposition \ref{propPhit} and Lemma \ref{lemmaHMVMsupportadd}, then $f_{\lambda_P,\lambda'}(\Phi_t)$ is proportional to $\Phi_t'$ and $r_N(\Phi_t)$ is proportional to $\varphi_t$.	
		
	\end{proposition}
	
	\begin{proof}
		
		Assume that $g\in G$ such that $$Ug\ol{J'}\subset\mrsupp(f_{\lambda_P,\lambda'}(\Phi)\rest_{\lambda'})\subset Ut\ol{U(\mfb)}\ol{U^1(\mfa)}\subset Ut\ol{U^1(\mfa)}\ol{\Jmax},$$ then in particular $\mrhom_{\ol{J'}\cap U^g}(\lambda',\psi^g)\neq 0$, implying that $\chi^g\rest_{\Honemax\cap U^g}=\theta_{\max}\rest_{\Honemax\cap U^g}$. Using Proposition \ref{propthetachieqconverse} we have $g\in U B^\times \Jmax=UB^\times J'$, thus $$g\in UB^\times J'\cap UtU(\mfb)U^1(\mfa)=U(B^\times\cap tU(\mfb)U^1(\mfa))J'=UtJ'.$$
		Thus we have shown that $\mrsupp(f_{\lambda_P,\lambda'}(\Phi_t)\rest_{\lambda'})=Ut\ol{J'}$, which finishes the proof of the first claim. 
		
		Notice that $\mrsupp(\Phi_t\rest_{\lambda_P})=Ut\ol{J_P}=Ut(\ol{J_P}\cap \ol{P})\subset U\ol{M}(J_P\cap N)$, using Proposition \ref{propGGmoduleJacquet}, we have $\mrsupp(r_N(\Phi_t)\rest_{\lambda_M})=U_Mt\ol{J_M}$, which finishes the proof the second claim.
		
	\end{proof}

	For integers $a\leq b$, we define 
	$$T(\mfb)_{[a,b]}:=\{\mrdiag(\varpi_E^{s_1}I_{m_0},\varpi_E^{s_2}I_{m_0}\dots, \varpi_E^{s_k}I_{m_0})\mid s_1,s_2,\dots,s_k\in \{a,a+1,\dots,b\}\}$$
	and $T(\mfb)_{\leq,[a,b]}=T(\mfb)_{[a,b]}\cap T(\mfb)_{\leq}$.
	
	\begin{lemma}\label{lemmaTwphit}
		
		\begin{enumerate}
			
			\item For $w\in\mfS_k$ and $t\in T(\mfb)_{[a,b]}$, we have $$T_w\ast t_P(\phi_{t^{-1}})=\sum_{w'\in\mfS_k,l(w')\leq l(w)}\sum_{t'\in T(\mfb)_{[a,b]}}c_{w,w';t,t'}t_P(\phi_{t'^{-1}})\ast T_{w'},$$
			where $c_{w,w';t,t'}$ are non-negative integers.
			
			\item Moreover, $c_{w,w';t,t'}\neq 0$ implies that $\ord{t'}=\ord{t}$, where we define $\ord{t_0}=s_1+s_2+\dots+s_k$ for any $t_0=\mrdiag(\varpi_E^{s_1}I_{m_0},\varpi_E^{s_2}I_{m_0}\dots, \varpi_E^{s_k}I_{m_0})\in T(\mfb)$ .
			
		\end{enumerate}

	\end{lemma}  
	
	\begin{proof}
		
		When $w$ is a simple reflection, it follows from the explicit multiplicative law of affine Hecke algebra (\emph{cf.} \eqref{eqbernstein}). In general, we do induction on the length $l(w)$.
		
	\end{proof}
	
	\begin{lemma}
		
		For $t\in T(\mfb)_{\leq,[a,b]}$ and $w\in \mfS_k$, we have $\mrsupp(r_N(T_w\ast\Phi_t))\subset U_MT(\mfb)_{[a,b]}\ol{J_M}$.
		
	\end{lemma}

	\begin{proof}
		
		Choose $\varphi_1,\varphi_t\in\mcv_M^{\lambda_M}$ supported on $U_M\ol{J_M}$ and $U_Mt\ol{J_M}$ respectively, and $\phi_{t^{-1}}\in \mch(\ol{M},\lambda_M)\cong\mca$ supported on $\ol{J_M}t^{-1}$, such that $\phi_{t^{-1}}\ast \varphi_1=\varphi_t$. Up to a scalar, we may also choose $\Phi_1,\Phi_t\in\mcv^{\lambda_P}$ supported on $U\ol{J_P}$ and $Ut\ol{J_P}$ respectively such that $r_N(\Phi_1)=\varphi_1$ and $r_N(\Phi_t)=\varphi_t$. 
		
		Using Lemma \ref{lemmaTwphit}, we have 
		\begin{equation*}
			\begin{aligned}
				r_N(T_w\ast \Phi_t)=r_N(T_w\ast t_P(\phi_{t^{-1}})\ast \Phi_1)&=\sum_{w'\in\mfS_k,l(w')\leq l(w)}\sum_{t'\in T(\mfb)_{[a,b]}}c_{w,w';t,t'}r_N(t_P(\phi_{t'^{-1}})\ast T_{w'}\ast\Phi_1)\\
				&=\sum_{w'\in\mfS_k,l(w')\leq l(w)}\sum_{t'\in T(\mfb)_{[a,b]}}(-1)^{l(w')}c_{w,w';t,t'}r_N(t_P(\phi_{t'^{-1}})\ast\Phi_1)\\
				&=\sum_{w'\in\mfS_k,l(w')\leq l(w)}\sum_{t'\in T(\mfb)_{[a,b]}}(-1)^{l(w')}c_{w,w';t,t'}\phi_{t'^{-1}}\ast\varphi_1
			\end{aligned}
		\end{equation*}
		where from the first row to the second row, by Corollary \ref{corPhi't}.(2), we have $T_{w'}\ast\Phi_1=(-1)^{l(w')}\Phi_1$. Using Lemma \ref{lemmaHMVMsupportadd}, $\phi_{t'^{-1}}\ast\varphi_1$ is a function in $\mcv_M^{\lambda_M}$ supported on $U_Mt'\ol{J_M}$. So the proof is finished.

	\end{proof}
	
	\begin{lemma}\label{lemmaV0n0-1basis}
		
		For $t_1',\dots,t_{u'}'$ ranging over $T(\mfb)_{\leq,[0,n_0-1]}$ and $\Phi_{t_1'},\dots,\Phi_{t_{u'}'}\in \mcv^{\lambda_P}$
		the related functions, let $\mcb(\mch_0\ast\Phi_{t_i'})$ be a $\mbc$-basis of $\mch_0\ast\Phi_{t_i'}$, then the disjoint union 
		$$\bigcup_{i=1}^{u'}r_N(\mcb(\mch_0\ast\Phi_{t_i'}))$$
		forms a $\mbc$-basis of the subspace $\mcv_{M,[0,n_0-1]}^{\lambda_M}$ of functions in $\mcv_M^{\lambda_M}$ supported on $U_MT(\mfb)_{[0,n_0-1]}\ol{J_M}$.
		
	\end{lemma}
	
	\begin{proof}
		
		By Lemma \ref{lemmaVlambdasupport} and Corollary \ref{corPhi't}, $\mcv_{M,[0,n_0-1]}^{\lambda_M}$ is of dimension  $n_0^k$, and  $\bigcup_{i=1}^{u}r_N(\mcb(\mch_0\ast\Phi_{t_i'}))$ is a set of linearly independent functions of cardinality $n_0^k$. 
		Using the above lemma, $\bigcup_{i=1}^{u}r_N(\mcb(\mch_0\ast\Phi_{t_i'}))\subset \mcv_{M,[0,n_0-1]}^{\lambda_M}$. So the proof is finished.
		
	\end{proof}
	
	Recall the definition of $T_0(\mfb)\subset T(\mfb,\varrho)$ in \S \ref{subsectionsimpletypes}.  Let $\mca_0=\mbc[T_0(\mfb)]\subset\mca=\mbc[T(\mfb,\varrho)]$. Via the isomorphism $\mca\cong \mch(\ol{M},\lambda_M)$, we may identify $\mca_0$ with the subspace of functions in $\mch(\ol{M},\lambda_M)$ that are supported on $\ol{J_M}T_0(\mfb)\ol{J_M}$. Using Lemma \ref{lemmaHMVMsupportadd}, the following proposition is clear.
	
	\begin{proposition}\label{propA0modbasis}
		
		The disjoint union	$\bigcup_{i=1}^{u'}r_N(\mcb(\mch_0\ast\Phi_{t_i'}))$ forms a free $\mca_0$-module basis of $\mcv_M^{\lambda_M}$.
		
	\end{proposition}

	\subsection{A general strategy}
	
	Now we would like to fully determine $\mcv^{\lambda_P}$ as an $\mch$-module. Our strategy is as follows:
	\begin{equation}\label{eqstrategy}
		\text{\parbox{.85\textwidth}{Find $t_1,\dots,t_u\in T(\mfb)_{\leq}$, such that
				\begin{itemize}
					\item The $\mfS_k$-orbit of $t_i$ is the same as the $\mfS_k$-orbit of its image in $T(\mfb)/T(\mfb,\varrho)$.
					\item  The disjoint union $\bigcup_{i=1}^{u}\mcb(\mch_0\ast\Phi_{t_i})$ forms a free $\mca$-module basis of $\mcv^{\lambda_P}$ for related $\Phi_{t_1},\dots,\Phi_{t_u}\in \mcv^{\lambda_P}$, or equivalently, $\bigcup_{i=1}^{u}r_N(\mcb(\mch_0\ast\Phi_{t_i}))$ forms a free $\mca$-module basis of $\mcv_M^{\lambda_M}$. Here, for each $i$ we let $\mcb(\mch_0\ast\Phi_{t_i})$ be a $\mbc$-basis of $\mch_0\ast\Phi_{t_i}$.
		\end{itemize}}}
	\end{equation}
	As a result, by Corollary \ref{corPhi't}.(2) and Proposition \ref{propPhitiPhi'tivarphiti} we necessarily have
	$$\mcv^{\lambda_P}\cong\bigoplus_{i=1}^{u}(\mca\otimes_\mbc\mch_0)\ast\Phi_{t_i}$$
	as $\mch=\mca\otimes_\mbc\mch_0$ modules, where each $(\mca\otimes_\mbc\mch_0)\ast\Phi_{t_i}$ is isomorphic to $\mca\otimes_\mbc\mch_0\otimes_{\mch_{\mco_{t_i}}}\varepsilon_{\mco_{t_i}}$. Thus Theorem \ref{thmmain} is verified.
	
	\subsection{Fulfillment of strategy \eqref{eqstrategy} for the Savin cover}
	
	To fulfill strategy \eqref{eqstrategy}, from now on we assume $\ol{G}$ to be either a Kazhdan-Patterson cover or the Savin cover. 
	
	In this case, it is direct to verify that the $\mfS_k$-orbit of $t_i'$ is the same as the $\mfS_k$-orbit of its image in $\mcx(\lambda)=T(\mfb)/T(\mfb,\varrho)$.
	
	Assume that $\ol{G}$ is the Savin cover. In this case, we have $T_0(\mfb)=T(\mfb,\varrho)$ and $\mca_0=\mca$. Then we may simply choose $u=u'$ and $t_i=t_i'$, then using Proposition \ref{propA0modbasis} the strategy \eqref{eqstrategy} is fulfilled.
	
	\subsection{Fulfillment of strategy \eqref{eqstrategy} for the Kazhdan--Patterson covers}
	
	Assume that $\ol{G}$ is a Kazhdan-Patterson cover. In this case, we have $T(\mfb,\varrho)=\pairangone{\zeta_E}T_0(\mfb)$ and $\mca$ is generated by $\mca_0$ and $\zeta_E$. In particular, $\zeta_E^m\in T_0(\mfb)$ if and only if $n_0/d_0$ divides $m$, thus $\mcx(\lambda)=T(\mfb)/T(\mfb,\varrho)$ is of cardinality $n_0^{k-1}d_0$.
	
	We relate $\zeta_E$ to a function $\phi_{\zeta_E}\in\mch(\ol{M},\lambda_M)$ supported on $\ol{J_M}\zeta_E\ol{J_M}$, then $t_P(\phi_{\zeta_E})$ is central as an element in $\mch(\ol{G},\lambda_P)$. 
	
	We define $\ord{\cdot}$ on $T(\mfb)$ as in Lemma \ref{lemmaTwphit}.(2).
	
	We consider an equivalence relation on  $T(\mfb)$: for $t_1,t_2\in T(\mfb)$, we write $t_1\sim_{T(\mfb,\varrho)} t_2$ if $t_1/t_2\in T(\mfb,\varrho)$. 
	
	\begin{lemma}
		
		Let $\ol{G}$ be the Kazhdan--Patterson cover $\ol{G}_{\mathrm{KP}}^{\bs{c}}$.  
		
		\begin{enumerate}
			
			\item We have $n_0/d_0=\mrgcd(n/l_0,2\bs{c}r+r-1)$ and $\mrgcd(n_0/d_0,k)=1$.
			
			\item For $t_1,t_2\in T(\mfb)$ such that $t_1\sim_{T(\mfb,\varrho)} t_2$, we have $t_1/t_2\in T_0(\mfb)$ if and only if $\ord{t_1}\equiv\ord{t_2}$ (mod $n_0$). 
			
			\item For each $t_i'\in T(\mfb)_{\leq,[0,n_0-1]}$, the related elements in $\mfS_k\cdot t_{i}'$ lie in different equivalence classes.
			
		\end{enumerate}
		
	\end{lemma}
	
	\begin{proof}
		
		We prove statement (1). By definition, we have $n_0=n/\mrgcd(n,l_0)=n/l_0$ and $d_0=n/\mrgcd(n,l_0(2\bs{c}r+r-1))$, and thus $n_0/d_0=\mrgcd(n/l_0,2\bs{c}r+r-1)$. Since $k$ divides $r$ and is relatively prime to $2\bs{c}r+r-1$, we have $\mrgcd(n_0/d_0,k)=1$.
		
		We prove statement (2). For $t\in T_0(\mfb)$ we have $n_0$ divides $\ord{t}$. Also, we have $\ord{\zeta_E}=kd_0$. Thus by (1), $\zeta_E^m\in T_0(\mfb)$ if and only if $n_0/d_0$ divides $m$, if and only if $n_0$ divides $\ord{\zeta_E^m}$. So statement (2) readily follows.
		
		Now we prove statement (3). Assume that $t_i'$ and $w\cdot t_i'$ are equivalent elements in $T(\mfb)_{[0,n_0-1]}$. Since $\ord{w\cdot t_i'}=\ord{t_i'}$, by (2) we have $w\cdot t_i'/t_i'\in T_0(\mfb)$. Since $T_0(\mfb)\cap T(\mfb)_{[0,n_0-1]}=\{1\}$, we must have $w\cdot t_i'=t_i'$.
		
	\end{proof}
	
	Now we let  $\{t_1,\cdots,t_u\}$ be a subset of $\{t_1',\dots,t_{u'}'\}$ consisting of elements $t'$, such that the remainder of $\ord{t'}$ modulo $n_0$ lies in $\{0,1,\dots,d_0-1\}$. 
	
	Using the above lemma, for $m=0,1,\dots,n_0/d_0-1$, the disjoint union $\bigcup_{i=1}^u(\mfS_k\cdot t_u)\zeta_E^m$ consists of elements $t'$, such that the remainder of $\ord{t'}$ modulo $n_0$ lies in $\{mkd_0,mkd_0+1,\dots,mkd_0+d_0-1\}$. Thus for different $m$ the related unions are pairwise disjoint. As a result,  $\bigcup_{i=1}^u\mfS_k\cdot t_u\subset T(\mfb)_{[0,n_0-1]}$ is a set of representatives of $T(\mfb)/T(\mfb,\varrho)$.
	
	We claim that
	$$\bigcup_{m=0}^{n_0/d_0-1}\bigcup_{i=1}^u\phi_{\zeta}^{-m}\ast r_N(\mcb(\mch_0\ast\Phi_{t_i}))$$ 
	forms a free $\mca_0$-module basis of $\mcv_M^{\lambda_M}$. Indeed from our construction, Lemma \ref{lemmaHMVMsupportadd}, Lemma \ref{lemmaTwphit} and Lemma \ref{lemmaV0n0-1basis}, the free $\mca_0$-module generated by $$\bigcup_{i=1}^u r_N(\mcb(\mch_0\ast\Phi_{t_i}))$$ consists of elements in $\mcv_M^{\lambda_M}$ supported on all the possible finite union of double cosets  $\cup_{t'}U_Mt'\ol{J_M}$, where the remainder of $\ord{t'}$ modulo $n_0$ lies in $\{0,1,\dots,d_0-1\}$. Thus, the free $\mca_0$-module generated by $$\bigcup_{i=1}^u\phi_{\zeta}^{-m}\ast r_N(\mcb(\mch_0\ast\Phi_{t_i}))$$
	consists of elements in $\mcv_M^{\lambda_M}$ supported on all the possible finite union of double cosets  $\cup_{t'}U_Mt'\ol{J_M}$, where the remainder of $\ord{t'}$ modulo $n_0$ lies in $\{mkd_0,mkd_0+1,\dots,mkd_0+d_0-1\}$. So the above claim is verified.
	
	Thus in this case, 
	$\bigcup_{i=1}^{u}r_{N}(\mcb(\mch_0\ast\Phi_{t_i}))$ forms a free $\mca$-basis and the strategy \eqref{eqstrategy} is fulfilled for $\{t_1,\cdots,t_u\}$ as above. As a result, Theorem \ref{thmmain} is proved.

	\section{Calculation of Gelfand--Graev dimension}
	
	Let $\ol{G}$ be either an $n$-fold Kazhdan--Patterson cover or the Savin cover of $G=\mrgl_r(F)$, and $(\ol{J},\lambda)$ a simple type of $\ol{G}$. Recall that $(\ol{J},\lambda)$ is a type of a discrete inertial equivalence class $\mfs$ of $\ol{G}$ (\cite{zou2023simple}*{\S 9.7}). Let $\pi$ be a genuine discrete series representation of $\ol{G}$ having inertial equivalence class $\mfs$. We would like to study the Gelfand--Graev dimension of $\pi$.
	
	\begin{proposition}\label{propdiscreteWhittaker}
		
		We have $\dim_{\mbc}\mrhom_{\ol{G}}(\mrind_U^{\ol{G}}(\psi),\pi)=\car{\mcx(\lambda)/\mfS_k}$.
		
	\end{proposition} 
	
	\begin{proof}
		Taking the functor $\bs{\mathrm{M}}_\lambda$, we have
		$$\dim_{\mbc}\mrhom_{\ol{G}}(\mrind_U^{\ol{G}}(\psi),\pi)=\dim_{\mbc}\mrhom_{\mch}(\mcv^\lambda,\pi^\lambda).$$
		Using \cite{zou2023simple}*{Proposition 9.10}, $\pi^\lambda$ is a one-dimensional $\mch$-module, so that as a $\mch_0$-module it is isomorphism to the sign character $\varepsilon_0$.
		
		Using Theorem \ref{thmmain} and Frobenius reciprocity, we further have
		\begin{equation}
			\begin{aligned}
				\dim_{\mbc}\mrhom_{\mch}(\mcv^\lambda,\pi^\lambda)&=\dim_{\mbc}\mrhom_{\mch}(\bigoplus_{\mco\in\mcx(\lambda)/\mfS_k }\mca\otimes_{\mbc}(\mch_{0}\otimes_{\mch_{\mco}}\varepsilon_{\mco}),\pi^\lambda)\\
				&=\dim_{\mbc}\mrhom_{\mch_0}(\bigoplus_{\mco\in\mcx(\lambda)/\mfS_k }(\mch_{0}\otimes_{\mch_{\mco}}\varepsilon_{\mco}),\varepsilon_0)\\
				&=\dim_{\mbc}\bigoplus_{\mco\in\mcx(\lambda)/\mfS_k }\mrhom_{\mch_{\mco}}(\varepsilon_{\mco},\varepsilon_0\rest_{\mch_{\mco}})\\
				&=\car{\mcx(\lambda)/\mfS_k}.
			\end{aligned}
		\end{equation}
	\end{proof}
	
	We remark that by direct calculation
	\begin{equation}\label{eqGGdim}
		\car{\mcx(\lambda)/\mfS_k}=\begin{cases}
			\binom{k+n_0-1}{k}d_0/n_0\quad&\text{if}\ \ol{G} \ \text{is a Kazhdan--Patterson cover;}\\
			\binom{k+n_0-1}{k}\quad&\text{if}\ \ol{G} \ \text{is the Savin cover.}
		\end{cases}
	\end{equation}
	Here, we denote by $\binom{b}{a}$
	the binomial coefficient of $x^a$ in $(1+x)^{b}$. 
	
	Since every discrete series representation corresponds to a simple type as above (\emph{cf.} \cite{zou2023simple}*{Proposition 9.12}), the above proposition calculates the Gelfand--Graev dimension of all the discrete series representations. 
	
	Using the Zelevinsky classification (\emph{cf.} \cite{kaplan2022classification}), theoretically we may calculate the Gelfand--Graev dimension for any irreducible representation, which should be a linear combination of terms as in \eqref{eqGGdim} with coefficients being plus or minus Kazhdan--Lusztig polynomials evaluated at 1 (\cite{lapid2018geometric}*{Section 10}).

\end{document}